\documentclass[english]{article}	
\usepackage{amssymb,amscd,amsmath,amsthm,color}
\usepackage[numbers,square,sort&compress]{natbib}
\newtheorem{theorem}{Theorem}[section]

\newtheorem{lemma}[theorem]{Lemma}
\newtheorem{proposition}[theorem]{Proposition}

\newtheorem{remark}[theorem]{Remark}

\usepackage{graphicx}

\DeclareMathOperator{\curl}{\mathbf{curl}}

\DeclareMathOperator{\Ddiv}{\mathrm{div}}

\def\calE{\boldsymbol{\mathcal{E}}}
\def\calH{\boldsymbol{\mathcal{H}}}
\def\calG{\boldsymbol{\mathcal{G}}}
\def\bfE{\mathbf{E}}

\def\bfp{\mathbf{p}}
\def\bfH{\mathbf{H}}

\def\bfL{\mathbf{L}}
\def\bf0{\boldsymbol{0}}

\def\bfG{\mathbf{G}}
\def\bfg{\mathbf{g}}
\def\bfnu{\boldsymbol{\nu}}
\def\bfxi{\boldsymbol{\xi}}

\def\bfx{\mathbf{x}}

\def\bfv{\mathbf{v}}
\def\bfw{\mathbf{w}}
\def\bfu{\mathbf{u}}
\def\bfp{\boldsymbol{p}}

\newcommand{\Mop}{\boldsymbol{\mathcal{M}}}
\renewcommand{\Re}{\mathrm{Re}\,}
\renewcommand{\Im}{\mathrm{Im}\,}

\newcommand{\R}{\mathbb{R}}
\newcommand{\C}{\mathbb{C}}

\newcommand{\XX}{\mathbf{X}}

\renewcommand*{\i}{\mathrm{i}}

\newcommand*{\ol}[1]{\overline{#1}}

\newcommand{\GammaI}{\Gamma_{\mathrm{I}}}
\newcommand{\GammaM}{\Gamma_{\mathrm{M}}}
\newcommand{\OmegaI}{\Omega_{\mathrm{I}}}
\newcommand{\OmegaM}{\Omega_{\mathrm{M}}}

\newcommand{\e}{\mathbf{e}}
\newcommand{\f}{\mathbf{f}}
\newcommand{\g}{\mathbf{g}}

\newcommand{\p}{\boldsymbol{p}}

\newcommand{\w}{\mathbf{w}}
\newcommand{\x}{\mathbf{x}}
\newcommand{\y}{\mathbf{y}}
\newcommand{\z}{\mathbf{z}}
\def\dcalE{\left({\calE}^+-{\calE}^-\right)}
\def\dcalH{\left({\calH}^+-{\calH}^-\right)}
\def\bfFe{\boldsymbol{F}_{\calE}}
\def\bfFh{\boldsymbol{F}_{\calH}}
\def\ddcalE{[[{\calE}]]}
\def\ddcalH{[[{\calH}]]}
\newcommand{\q}{\mathbf{q}}

\usepackage{fullpage}

\begin{document}

\title{The Time Domain Linear Sampling Method for determining the shape of multiple scatterers using electromagnetic waves}

\author{Timo L\"ahivaara$^a$, Peter Monk$^b$, Virginia Selgas$^{c,*}$\bigskip\\
  ${}^a$Department of Applied Physics, University of
  Eastern Finland,\\ 70211 Kuopio, Finland\\
  ${}^b$Department of Mathematical Sciences,
  University of Delaware,\\ Newark, DE 19716, USA\\
  ${}^c$Departamento de Matem\'aticas, Universidad de
  Oviedo,\\ EPIG, C/ Luis Ortiz Berrocal s/n, 33203 Gij\'on, Spain\\
${}^*$Corresponding author}

\maketitle
\noindent
\textit{Dedicated to the memory of our dearest colleague and friend Francisco Javier
  Sayas. Your passion for retarded layer potentials inspired our
  work.}

\section*{Abstract}

The time domain linear sampling method (TD-LSM) solves inverse scattering problems using time domain data by creating an indicator function for the support of the unknown scatterer. It involves only solving a
linear integral equation called the near-field equation using different data from sampling points that probe the domain where the scatterer is located.  To date,
the method has been used for the acoustic wave equation and has been tested for several different types of scatterers, i.e. sound hard, impedance, and
penetrable, and for waveguides.  In this paper, we extend the TD-LSM to the time dependent Maxwell's system with impedance boundary conditions - a
similar analysis handles the case of a perfect electric conductor (PEC).  We provide an analysis that supports the use of the TD-LSM for this problem, and preliminary numerical tests of the algorithm.  Our analysis
relies on the Laplace transform approach previously used for the acoustic wave equation. This is the first application of the TD-LSM in electromagnetism.\bigskip

\noindent
\textbf{keywords}: Wave-based imaging, electromagnetism, impedance, linear sampling, time domain

\section{Introduction}\label{sec:intro}
The inverse scattering problem studied in this paper concerns the 
reconstruction of the shape of a bounded scatterer using time domain electromagnetic scattering data. In particular, we probe the scatterer using incident fields due to point sources located away from the scatterer, and the data for the inverse problem is the
scattered field measured on a surface containing the unknown scatterer which could have multiple components.  Reconstructing the shape of the scatterer from this data is a non-linear and ill-posed problem.

In the frequency domain, there are many possible techniques to solve this problem.  For example, optimization based schemes determine the unknown shape by finding the best fit of the data using a suitable parametrization of the unknown surface.  Obviously such a method requires a priori knowledge of the nature of the scatterer, including the number of scatterers
and their topology.  This is a very flexible technique able to handle many different measurement data (multistatic, bistatic etc), although generally computationally intensive.  For theoretical progress and results for electromagnetic inverse problems using real measured data, see e.g. \cite{XChen18,belk01,belk04,lit09,dorn10,dido19}.  

In order to decrease the need for a priori data and mitigate the computational burden, an alternative approach is to use a \emph{qualitative method} which seeks to determine the shape of the scatterer without determining its material
properties.  This type of approach started with the work of Colton and Kirsch \cite{col96} for the Helmholtz equation and has been expanded to include a variety of methods, as well as applications to electromagnetism and elasticity.   
Of particular relevance to our work is the \emph{Linear Sampling Method} (LSM)~\cite{ccm-book}, {\cite[Chapter 14]{MonkBook} and \cite[Pages 191-240]{CIME}} for Maxwell's equations (see also {the latter} for the related Generalized Linear Sampling Method).  This method solves a sequence of linear integral equations to construct an indicator function that can be used to determine the boundary of the scatterer.

While the frequency domain LSM is certainly easy to implement and 
can determine the shape of the scatterer for a variety of different types of scatterer, it requires a large amount of multi-static data. In an effort to decrease the number of source and receiver points, the
Time Domain LSM (TD-LSM) was proposed, and analyzed  for acoustic scattering by a sound soft object by Haddar et al.~\cite{CHLM10}.
Numerical results show that a coarser set of data can be used.  The TD-LSM was
then extended, with an improved numerical implementation, to scatterers with
an impedance boundary condition by Marmorat et al.~\cite{MarmoratEtal}. 
Following encouraging numerical results~\cite{GuoEtal}, an analysis of the
TD-LSM for
penetrable acoustic scatterers was given by Cakoni at al.~\cite{CMS-21}.  
This analysis is based on the use of Laplace transforms to prove continuity
estimates for the method and relies upon the localization of transmission
eigenvalues due to Vodev~\cite{Vodev-18}.

Our paper is devoted to extending the TD-LSM to the time 
domain electromagnetic inverse problem.  We give the first analysis and
preliminary numerical results for the method. The methods we use to prove the theoretical results
are extensions of the techniques in~\cite{CHLM10,MarmoratEtal,monksel16}
via Laplace transforms.  We choose to analyze scattering by an impenetrable
scatterer with an impedance boundary condition (suitable for an imperfect
conductor).  The same analysis can be adapted to the case of a perfect conductor and we shall show a numerical example of this case.  Unfortunately, for Maxwell's equations,  
it is not currently possible to analyze the case of a 
penetrable scatterer using
Laplace transforms: 
As shown by Cakoni et al.~\cite{cak21} for the
case of a spherically stratified medium, and by Vodev~\cite{Vodev-21} for more general scatterers, there does not exist a suitable half-plane of the complex plane
that is devoid of transmission eigenvalues, and this rules out the simple
use of Laplace transforms as was done for the Helmholtz equation~\cite{CMS-21}.
For a complete discussion of this issue, see~\cite{cak21}. 
Despite being unable to analyze the TD-LSM for penetrable scatterers, we can
still test the algorithm in this case.  In Section~\ref{sec:num} we provide
an example with a penetrable scatterer.  The reconstruction is very similar to
the case of perfect conducting or impedance scatterers, so suggesting that
the TD-LSM may be applicable in that case.

To test the TD-LSM, we use synthetic data computed via a nodal Discontinuous Galerkin (DG) method coupled with a low-storage explicit Runge-Kutta time stepping scheme \cite{hesthaven_warburton_book, carpenter94}. The DG method provides an efficient technique to numerically solve differential equations and has properties that make it well-suited for wave simulations, see e.g. \cite{HESTHAVEN2002186, WILCOX20109373, LI2014915, tl19}. These features include, e.g., high-order accuracy, straightforward handling of large discontinuities in the material parameters, and support for complex problem geometries. In addition, the method has excellent parallelization properties in both CPU and GPU environments, see e.g. \cite{Klockner09, godel10, Melander}. All of these are essential features for the numerical scheme to be used for solving complex wave problems.

The layout of this paper is as follows: in Section~\ref{sec:fwdproblem} we give details of the time dependent forward problem for electromagnetic scattering from a bounded scatterer with an impedance boundary condition, and derive relevant continuity estimates for the problem. In Section~\ref{sec:LSM} we formulate the TD-LSM for Maxwell's equations and prove analogues of the usual theorems regarding the performance of the TD-LSM.  In Section~\ref{sec:num}
we provide some details of the implementation of the inversion technique and 
then give some  numerical examples showing the performance of the method.

Throughout the paper, boldface font will be used to represent vector quantities.  For example
\[
\bfL^2(\Omega^c)=\{\bfu\;|\; \bfu=(u_1,u_2,u_3)^T\mbox{ and } u_{\i}\in L^2(\Omega^c),\, \i=1,2,3\} \, .
\]
We define $\mathbb{R}_+=\{x\in\mathbb{R}\;|\; x>0\}$. For a Hilbert space $X$, we consider $\mathcal{D}'(\mathbb{R},X)$ the space of $X$-valued distributions depending on $t\in\mathbb{R}$; we denote by $\mathcal{D}'(\mathbb{R}_+,X)$ the subset of causal distributions, that is, those distributions which vanish on $(-\infty,0)$.


\section{Maxwell's equations with an impedance boundary condition}\label{sec:fwdproblem}
We start by defining some notation and spaces.
Let $\Omega$ denote the scatterer. It is assumed to be a bounded {and possibly multiconnected} domain whose complement $\Omega^c=\mathbb{R}^3\setminus \overline{\Omega}$ is connected and whose boundary $ \Gamma=\partial  \Omega$ is  Lipschitz continuous. We denote by $\bfnu$ the unit outward normal to $\Omega$ on $\Gamma$.

The appropriate solution space in the Laplace domain will be a subspace of \[
\bfH (\curl,\Omega^c)=\{\bfw\in \bfL^2(\Omega^c) \; | \; \curl\bfw\in \bfL^2(\Omega^c) \}\, .\] 
To define such subspace we need the space of tangential 
square integrable functions on $\Gamma$:
\[
\bfL^2_T(\Gamma)=\{\bfu\in (L^2(\Gamma))^3\;|\; \bfnu\cdot\bfu=0\mbox{ a.e. on }\Gamma\}\, .
\]
Similarly, we consider
\[
\bfH^{\pm 1/2}_T(\Gamma)=\{\bfu\in (H^{\pm1/2}(\Gamma))^3\;|\; \bfnu\cdot\bfu=0\mbox{ a.e.  on }\Gamma\} \, .
\]

\subsection{Time domain Maxwell's equations}
In the time domain, Maxwell's equations for the causal electric field $\tilde{\calE}$ and magnetic field $\tilde{\calH}$ in $ \Omega^c \times \mathbb{R} $  are
\begin{eqnarray*}
\tilde{\epsilon}\,\tilde{\calE}_t-\curl\tilde{\calH}&=&\bf0\, ,\\
\tilde{\mu}\,\tilde{\calH}_t+\curl\tilde{\calE}&=& \bf0\, ,
\end{eqnarray*}
where we assume that there is no imposed current and that the material that fills $\Omega^c$ is lossless (e.g. air or vacuum). {Here we denote the spatially dependent electric permittivity and magnetic permeability by $\tilde{\epsilon}$ and $\tilde{\mu}$, respectively. Their counterparts  in vacuum are  denoted by $\epsilon_0$ and $\mu_0$.} We  suppose that the relative electric permittivity $\epsilon_r=\epsilon_0^{-1}\tilde{\epsilon}:\Omega^c\to\R^{3\times 3}$ and {relative} magnetic permeability $\mu_r=\mu_0^{-1}\tilde{\mu}:\Omega^c \to\R^{3\times 3}$ are symmetric matrix-valued functions with uniformly bounded entries that are piecewise in {$W^{1,\infty}(\Omega^c)$ and that $\epsilon_r=I_{3}$ and $\mu_r=I_{3}$ sufficiently far from the scatterer ($I_{3}$ stands for the $3\times 3$ identity matrix).} They are also assumed to be uniformly positive definite almost everywhere in $\Omega^c$. 

The electromagnetic field is assumed to be subject to the following impedance boundary condition, that models, for example, an imperfectly conducting body:
\[
\tilde{\calH}\times\bfnu+\tilde{\Lambda}\tilde{\calE}_T=\tilde{\calG}\quad \mbox{on }  \Gamma \times \mathbb{R} \, .
\]
Here we denote by $\tilde{\calE}_T=(\bfnu\times\tilde{\calE})\times\bfnu$ the tangential trace of the electric field $\tilde{\calE}$, and this notation will be used in the sequel for such a trace of any smooth enough vector field. Concerning the datum $\tilde{\calG}$, it is a causal tangential vector field that is usually obtained from the trace of a smooth incident field, as we will detail in the following sections. Moreover, the matrix-valued function $\tilde{\Lambda}:\Gamma\to\mathbb{R}^{3\times 3}$
is assumed to be uniformly bounded and symmetric almost everywhere on $\Gamma$. We also assume that  $\tilde{\Lambda}$ maps tangential vectors to tangential vectors, for which it is uniformly positive definite. By this we mean that there exists $\tilde{\Lambda}_{\min}>0 $ such that, for almost every $\bfx$ on $\Gamma$ and for each vector $\bfxi\in\mathbb{R}^3$ that is
tangential to $\Gamma$ at $\bfx$ (i.e. $\bfxi\cdot\bfnu(\bfx)=0$), also the vector $\tilde{\Lambda}(\bfx)\bfxi$ is tangential to $\Gamma$ at $\bfx$ (i.e. $\tilde\Lambda(\bfx)\bfxi\cdot\bfnu(\bfx)=0$) and it holds that $\tilde{\Lambda}(\bfx)\bfxi\cdot\overline{\bfxi}\geq \tilde{\Lambda}_{\min}\,|\bfxi|^2$. 

The speed of light in vacuum is given by $c_0=(\epsilon_0\mu_0)^ {-1/2}$. Then, following \cite[Section 6.1]{ColtonKress}, we  rescale the electric and magnetic fields:
\[
 \calE=\epsilon_0^{1/2}\tilde{\calE}\qquad \mbox{and}\qquad
 \calH=\mu_0^{1/2}\tilde{\calH}\, .
\]
The rescaled electromagnetic field is still causal and satisfies
\begin{eqnarray*}
c_0^{-1}{\epsilon}_r\,\calE_t-\curl\calH =\bf0  &\qquad\text{in } \Omega^c \times \mathbb{R} \, ,\\
c_0^{-1}\mu_r\,\calH_t+\curl\calE  =\bf0 &\qquad\text{in } \Omega^c \times \mathbb{R}\, ,
\end{eqnarray*}
and is subject to the impedance boundary condition 
\begin{equation}
\label{eq:sma_1}
\calH\times\bfnu+\Lambda\calE_T=\calG\qquad\text{on } \Gamma \times \mathbb{R} \, .
\end{equation}
Above we have set $\Lambda=Z_0\tilde{\Lambda}$ and $\calG=\mu_0^{1/2}\tilde{\calG}$, where   $Z_0=(\mu_0/\epsilon_0)^{1/2}$ is the impedance of free space.  Notice that there is no need for a radiation condition in the time domain under the causality assumption. This is because of the finite speed of propagation of electromagnetic waves, so that at any time $t$ there is a large enough ball in $\mathbb{R}^3$ for which the field vanishes outside of this ball.

The problem is typically rewritten in terms of either the rescaled electric or magnetic field. Here we opt for the former: More precisely, we rewrite the second equation as $\calH_t= -c_0 \mu_r^{-1}\curl\calE$ and use this in the time derivative of  the remaining equations to obtain 
\begin{eqnarray}
c_0^{-2}{\epsilon}_r\calE_{tt}+\curl 
(\mu_r^{-1}\curl\calE) \, =\, \bf0 \qquad & \mbox{in } \Omega^c \times \mathbb{R}\, , \label{eq:TDfwd_Omegac}\\
(\mu_r^{-1}\curl\calE)\times\bfnu-c_0^{-1}{\Lambda}\calE_{t,T} \, =\, -c_0^{-1}\calG_t \qquad & \mbox{on }\Gamma \times \mathbb{R}\, . \label{eq:TDfwd_Gamma}
\end{eqnarray}

\subsection{Analysis of the forward problem based on the Fourier-Laplace transform}

Let us first recall some basic facts about the Fourier-Laplace transform, cf.~\cite{Lubich, MarmoratEtal} and \cite[Chapters 2-3]{Sayas2016}, that will be used here.
For a Banach space $X$, let $\mathcal{D}'(\R;X)$ and $\mathcal{S}'(\R;X)$ represent the space of $X$-valued distributions and tempered
distributions on the real line, respectively. For any $s_0\in\R$, $s_0>0$, we set
$\mathcal{L}'_{s_0}(\R;X) = \{ f\in\mathcal{D}'(\R;X) \;|\; e^{-s_0t}f\in \mathcal{S}'(\R;X)\}$.  This allows us to consider the Laplace transform of any   $f\in \mathcal{L}'_{s_0}(\R;X)$ such that $e^{-s_0t}f\in {L}^1(\R;X)$, defined by 
$$
\mathcal{L} [f] (s) = \int_{-\infty}^{\infty} e^{ist} f(t) dt \qquad \mbox{for a.e. } s\in\C_{s_0}\, ,
$$
where $\mathbb{C}_{s_0} = \{ s\in\mathbb{C}\;|\;\Im ( s) > s_0 \} $. 
In particular, when $X$ is a Sobolev space (e.g. $X=\bfH(\curl,\Omega^c)$), for any $p\in\mathbb{R}$ we consider the Hilbert space 
$$
H_{s_0}^{p}(\mathbb{R}; X)=\left\{ f\in \mathcal{L}'_{s_0}(\R;X) \;\big{|} \; \displaystyle\int_{-\infty+is_0}^{\infty+is_0} |s|^{2p} ||\mathcal{L}[f](s)||_X^2\, ds <\infty \right\}
$$
endowed with the norm $||f||_{H_{s_0}^{p}(\mathbb{R}; X)}=\Big(\!\displaystyle\int_{-\infty+is_0}^{\infty+is_0} |s|^{2p} ||\mathcal{L}[f](s)||_X^2\, ds \Big)^{1/2}$. We will make use of the well-known Plancherel's theorem, which relates the norm of a function $f$ in $H^p_{s_0}(\R;X)$ with the weighted norms of $\mathcal{L}[f](s)$ in $X$. Indeed, {rewriting the Fourier-Laplace transform in terms of the usual Fourier transform,} Plancherel's theorem  leads to
\begin{equation}\label{eq:Plancherel}
\Vert \mathcal{L} [f] (s)\Vert_{X} =\Vert \mbox{{$t\mapsto$}}e^{-\Im (s) t}f(t)\Vert _X\, ;
\end{equation}
in some situations, this is useful to deduce  bounds of time dependent fields, e.g.
$$
 \int_{-\infty+is_0}^{+\infty+is_0}  \| \mathcal{L}[f](s)\|^2_{X} \, ds = \int_{-\infty}^{+\infty} e^{-2s_0t}\|f(t) \|_X^2 \, dt  = \Vert f\Vert^2_{L^2_{s_0}(\R; X)}\, .
$$
We will make use of the Fourier-Laplace transform and get information back to the time domain thanks to the following result, see \cite{Lubich} and \cite[Chapter 3]{Sayas2016}.
\begin{lemma}\label{lem:lubich}
We consider two Banach spaces $X$ and $Y$, and write $\mathcal{B}(X,Y)$ to represent the space of linear and bounded operators from $X$ into $Y$. Let $s\in \mathbb{C}_{s_0} \mapsto f_s\in \mathcal{B}(X,Y)$ be an analytic function for which there exist $r\in\mathbb{R}$ and $C>0$ such that
$$
\Vert f_s\Vert _{\mathcal{B} (X;Y )} \leq C |s|^r \quad\mbox{ for a.e. } s\in \mathbb{C}_{s_0}\, .
$$
Set $F(t) = \displaystyle\frac{1}{2\pi} \displaystyle\int_{-\infty+is_0}^{+\infty+is_0} e^{-ist}\, f_s\, ds$, and $\mathcal{F}g(t) = (F*g)(t) = \displaystyle\int_{-\infty}^{+\infty} F(\tau) g(t-\tau)\, d\tau $ the associated convolution
operator. Then, for all $p \in\mathbb{R}$, $\mathcal{F}$ extends to a bounded operator from $H_{s_0}^{p+r}(\mathbb{R}; X)$ to $H^p_{s_0}(\mathbb{R};Y )$.
\end{lemma}
In all the sequel, we will replace $\R$ by $\R_+$ in the above defined spaces in order to denote the corresponding subspaces of causal functions. In this sense, Paley-Wierner theory will be frequently applied to study casuality, cf. \cite[Sections 2.1 and 3.1]{Sayas2016}.

We next use the Fourier-Laplace transform to study the forward problem at hand. More precisely, when we formally take such a transform in the time domain equations (\ref{eq:TDfwd_Omegac}-\ref{eq:TDfwd_Gamma}) we get
\begin{eqnarray}
\curl (\mu_r^{-1}\curl\bfE_s)-c_0^{-2}{s^2}{\epsilon}_r\bfE_{s} \, =\, \bf0 \qquad & \mbox{ in } \Omega^c\, , \label{eq_FL_1}\\
(\mu_r^{-1}\curl\bfE_s)\times\bfnu+isc_0^{-1} {\Lambda} \bfE_{s,T} \, =\, ic_0^{-1}{s}\bfG_s\qquad&\mbox{on }\Gamma\, , \label{eq_FL_2}
\end{eqnarray}
where we denote $\bfE_s(\bfx)=\mathcal{L}[\calE (t,\bfx)](s)$ and $\bfG_s(\bfx)=\mathcal{L}[\calG (t,\bfx)](s)$. 
Notice that, provided $\Im(s)>s_0>0$, there is no need to impose a Silver-M\"uller radiation condition in the Fourier-Laplace domain fields but it suffices to require
$\bfE_s\in \bfH (\curl;\Omega^c)$. Also notice that  $\bfH_s(\bfx)=\mathcal{L}[\calH (t,\bfx)](s)$  can be recovered from $\bfE_s$ by taking the Fourier-Laplace transform of  $\calH_t= -c_0 \mu_r^{-1}\curl\calE$, which leads to  $-is\bfH_s= -c_0 \mu_r^{-1}\curl\bfE_s$.

Next we obtain a variational formulation of the scattering problem (\ref{eq_FL_1}-\ref{eq_FL_2}). 
To this end, we multiply both sides of
equation (\ref{eq_FL_1}) by the complex conjugate {$\overline{\bfv}$} of a smooth test
function $\bfv$ of compact support and
integrate by parts in $\Omega^{c}$ to obtain:
\begin{eqnarray*}
0 
& = &  \int_{\Omega^c} \left( (\mu_r^{-1}\curl\bfE_s)\cdot \curl \overline{\bfv} - k_s^2\,  {\epsilon}_r  \bfE_{s} \cdot \overline{\bfv} \right) dV 
-\int_{\Gamma}\bfnu\times (\mu_r^{-1}\curl\bfE_s ) \cdot \overline{\bfv} \, dA \, ,
\end{eqnarray*}
where $k_s=\displaystyle\frac{s}{c_0}$. 
We make use of the impedance boundary condition to rewrite the integral on $\Gamma$ as
\begin{eqnarray*}
-\int_{\Gamma}\bfnu\times (\mu_r^{-1}\curl\bfE_s ) \cdot \overline{\bfv} \, dA = 
ik_s\int_{\Gamma}
 (- \Lambda\bfE_{s,T} + \bfG_s )  \cdot \overline{\bfv}_T \, dA  \, .
\end{eqnarray*}
Now we need to define the solution space
\[
\XX=\{\bfv\in \bfH(\curl,\Omega^c)\, |\, \bfv_T\in \bfL^2_T(\Gamma)\}
\]
endowed with the norm {$\Vert \bfv\Vert_\XX=(\Vert \bfv \Vert_{\bfH(\curl,\Omega^c)}^2+\Vert \bfv_T\Vert_{\mathbf{L}^2(\Gamma)}^2)^{1/2} $}. Using the density of {$\boldsymbol{\mathcal{C}}_0^\infty(\overline{\Omega^c})$} in $\XX$ (cf. {\cite[Th. 3.54]{MonkBook} for a bounded Lipschitz domain}), the variational form of the Fourier-Laplace domain forward problem is then to find $\bfE_s\in \XX$ that satisfies
\begin{eqnarray*}
\int_{\Omega^c} \left( (\mu_r^{-1}\curl\bfE_s)\cdot \curl \overline{\bfv} - k_s^2\,   {\epsilon}_r \bfE_{s} \cdot \overline{\bfv} \right) dV
-ik_s\int_{\Gamma} \Lambda\bfE_{s,T} \cdot \overline{\bfv}_T \, dA 
=
-ik_s \int_{\Gamma} \bfG_{s}  \cdot \overline{\bfv}_T\, dA\, 
\end{eqnarray*}
for any  $\bfv\in \XX$.

In order to study this variational formulation, we suppose that $s\in \mathbb{C}_{s_0}$ for some fixed $s_0>0$ and consider the sesquilinear form associated to the left-hand side:
\begin{eqnarray*}
a_s(\bfu,\bfv)=\int_{\Omega^c} \left( \left(\mu_r^{-1}\curl\bfu\right)\cdot \curl \overline{\bfv} - k_s^2\, \epsilon_r  \bfu \cdot \overline{\bfv} \right) dV
 - i k_s \int_{\Gamma} \Lambda\bfu_{T} \cdot \overline{\bfv}_T \, dA \, .
\end{eqnarray*}
Using the approach in \cite{BHa86}, notice that
\begin{eqnarray*}
-\overline{s} a_s(\bfv,\bfv)=\int_{\Omega_c} \left( \left(-\overline{s} \mu_r^{-1}\curl\bfv\right)\cdot \curl \overline{\bfv} + s\, \left(\frac{|s|}{c_0}\right)^2  {\epsilon}_r  \bfv \cdot \overline{\bfv} \right) dV
 + i \frac{|s|^2}{c_0}  \int_{\Gamma} \Lambda\bfv_{T} \cdot \overline{\bfv}_T\, dA \, .
\end{eqnarray*}
The imaginary part of this expression can be studied term by term under our assumptions on the coefficients to obtain the following inequality:
\begin{eqnarray*}
\Im \left(-\overline{s} a_s(\bfv,\bfv)\right) & \geq& s_0 \mu_{r,\max}^{-1} \Vert\curl\bfv\Vert_{0,\Omega^c}^2 + \displaystyle \left(\frac{|s|}{c_0}\right)^2  s_0  \epsilon_{r,\min} \,  \Vert \bfv\Vert _{0,\Omega^c}^2 + \frac{|s|^2}{c_0}\Lambda_{\min} \Vert \bfv_T\Vert ^2_{0,\Gamma}\, \\
&\geq&\min\left(s_0 \mu_{r,\max}^{-1},\displaystyle \left(\frac{|s|}{c_0}\right)^2  s_0  \epsilon_{r,\min} ,\frac{|s|^2}{c_0}\Lambda_{\min}\right) \Vert \bfv\Vert_{\XX}^2 \, .
\end{eqnarray*}
Here and in the sequel, $\mu_{r,\max}^{-1}>0$, $\epsilon_{r,\min}>0$ and $\Lambda_{\min}=Z_0\tilde{\Lambda}_{\min} >0$ are positive constants associated with the positive definiteness properties of the coefficient functions $\mu_{r}$, $\epsilon_{r}$ and $\Lambda$, respectively.
Also notice that, for any fixed $s\in\C $, the boundedness of the sesquilinear form $a_s:\XX\times \XX\to\C$ follows from the assumed uniform boundedness of the coefficient functions and the definition of the space $\XX$. Therefore, the Lax-Milgram lemma guarantees that there exists a unique solution $\bfE_s\in \XX $ such that
$$
a_s(\bfE_s ,\bfv)\, =\,
-i\frac{s}{c_0} \int_{\Gamma} \bfG_{s}  \cdot \overline{\bfv}_T \, dA \quad\forall \bfv\in\XX \, .
$$
To allow us to go back to the time domain problem, we need bounds on $\bfE_s$ which make explicit the dependence on $s$. With this aim, we notice that, since
$$
\Im(-\overline{s}a_s(\bfE_s ,\bfE_s))\, =\,
\Im\left(i\,\frac{|s|^2}{c_0} \int_{\Gamma} \bfG_{s}  \cdot \overline{\bfE}_{s,T}\, dA \right) = \frac{|s|^2}{c_0}\, \Re \left(\int_{\Gamma} \bfG_{s}  \cdot \overline{\bfE}_{s,T}\, dA \right) ,
$$
it follows that
$$
s_0\, \mu_{r,\max}^{-1} \Vert \curl\bfE_s\Vert _{0,\Omega^c}^2 + \displaystyle \big(\frac{|s|}{c_0}\big)^2  s_0  \epsilon_{r,\min}  \,  \Vert \bfE_{s}\Vert _{0,\Omega^c}^2 + \frac{|s|^2}{c_0}\Lambda_{\min} \Vert \bfE_{s,T}\Vert ^2_{0,\Gamma}\,\leq \,
\frac{|s|^2}{c_0} \, \Re \left(\int_{\Gamma} \bfG_{s}  \cdot \overline{\bfE}_{s,T}\, dA \right) .
$$
By the Cauchy-Schwarz  inequality, we get the bound
$$
s_0 \,\mu_{r,\max}^{-1} \Vert \curl\bfE_s\Vert _{0,\Omega^c}^2 + \displaystyle \big(\frac{|s|}{c_0}\big)^2   s_0 \, \epsilon_{r,\min} \,  \Vert \bfE_{s}\Vert _{0,\Omega^c}^2  +\frac{|s|^2}{2c_0}\,\Lambda_{\min} \,\Vert \bfE_{s,T}\Vert ^2_{0,\Gamma}\,\leq \,
\frac{|s|^2}{2c_0}\,\Lambda_{\min}^{-1}\, \Vert \bfG_{s}\Vert ^2_{0,\Gamma}\,  ,
$$
which allows us to apply Lemma \ref{lem:lubich} to go back to the time domain and guarantee the following result. Notice that causality preservation is straightforward by the Paley-Wiener theory.

We have thus proved the following result:
\begin{lemma}\label{th:FwdProbOK}
For any $p\in\mathbb{R}$ and $\boldsymbol{\mathcal{G}}$ in $H^p_{s_0}(\mathbb{R};\bfL^2_T(\Gamma ))$, there exists a unique solution $\boldsymbol{\mathcal{E}}$ to (\ref{eq:TDfwd_Omegac}-\ref{eq:TDfwd_Gamma}), which is  bounded in  $H^p_{s_0}(\mathbb{R};\bfL^2(\Omega^c))$ and $H^{p-1}_{s_0}(\mathbb{R};\bfH(\curl,\Omega^c))$ in terms of the datum $\boldsymbol{\mathcal{G}}\in H^p_{s_0}(\mathbb{R};\bfL^2_T(\Gamma ))$. Its tangential trace lies in $H^p_{s_0}(\mathbb{R};\bfL_T^2(\Gamma))$. Moreover, causality is preserved: if $\boldsymbol{\mathcal{G}}$ belongs to $H^p_{s_0}(\mathbb{R}_{+};\bfL^2_T(\Gamma ))$, then the solution $\boldsymbol{\mathcal{E}}$ is in $H^p_{s_0}(\mathbb{R}_{+};\bfL^2(\Omega^c))$ and in $H^{p-1}_{s_0}(\mathbb{R}_{+};\bfH(\curl,\Omega^c))$. 
\end{lemma}

We define the solution operator $\mathcal{Q}$ that maps the datum  $\boldsymbol{\mathcal{G}}_t$ onto the solution $\boldsymbol{\mathcal{E}}$ of problem (\ref{eq:TDfwd_Omegac}-\ref{eq:TDfwd_Gamma}). So defined, we have just shown that
$$
\mathcal{Q}:\, \boldsymbol{\mathcal{G}}_t\in H^p_{s_0}(\mathbb{R};\bfL^2_T(\Gamma )) \mapsto \calE \in H^{p+1}_{s_0}(\mathbb{R};\bfL^2(\Omega^c))\cap H^{p}_{s_0}(\mathbb{R};\XX) 
$$
is bounded for all $p\in\R$, and preserves causality.

\section{The inverse problem and the linear sampling method}
\label{sec:LSM}

We now formulate precisely the inverse problem we shall study. We assume that the unknown scattering object $\Omega$ is illuminated by incident fields that are due to regularized point sources (see (\ref{EI_TD_1}) below)
which are a model of a source of electromagnetic waves. Each source point is placed on a fixed surface $\GammaI \subset \R^3$.
We seek to reconstruct the scatterer $\Omega$ from measurements of the scattered fields corresponding to those incident fields on a possibly different surface
$\GammaM \subset \R^3$, which is a  model for a measurement device.
Both $\GammaI$ and $\GammaM$ are piecewise smooth surfaces, and are allowed to be open or closed.  When either surface is closed we assume $\Omega$ is enclosed by that surface. In the case when $\GammaI$ (or $\GammaM$) is open, we suppose that it is a  subset of an analytic closed surface $\tilde{\Gamma}_{\mathrm{I}}$ (or $\tilde{\Gamma}_{\mathrm{M}}$, respectively) that encloses $\Omega$. {In our analysis it will be useful to let $\OmegaI$ denote the domain enclosed by $\GammaI$ (or the domain enclosed by the analytic  surface $\tilde{\Gamma}_{\rm{}I}$ containing $\GammaI$ when $\GammaI$ is open).  Similarly $\OmegaM$ is defined in the same way using $\GammaM$ in place of $\GammaI$.}

In order to describe the regularized point sources that we consider, we fix a polarization $\p\in\R^3\setminus\{\boldsymbol{0}\}$, a source point $\y \in \GammaI$, and a smooth
function $\chi \in \mathcal{C}^\infty_0(\R_+)$ that models a modulation function in time. Then we take regularized incident magnetic dipoles defined in the classical sense for
$\x \not = \y$ and $t \in \R$ by:
\begin{equation}
  \mathcal{E}^i(\x,t;\y,\p) = \curl_{\x}
    \left( \p\, \Phi_{\chi} (\x-\y,t)\,  \right)
  = - \p \times \nabla_{\x} \Phi_{\chi} (\x-\y,t) \, . \label{EI_TD_1}
\end{equation}
Here
\[
\Phi_{\chi}(\x,t)=\displaystyle\frac{\chi(t-c_0^{-1}|\x|)}{4\pi|\x|}
\]
is the regularized counterpart of the fundamental solution of the time dependent wave equation \[
\Phi(\x,t) = \displaystyle\frac{\delta(t-c_0^{-1}|\x|)}{4\pi|\x|}\,,
\] 
where $\delta$ denotes the Dirac delta distribution. In particular, these dipoles are divergence free away from the sources, that is, $\Ddiv_{\x}  \mathcal{E}^i(\x,t;\y,\p) = 0$ for $(\x,t)\in(\R^3\setminus\{ \y\})\times\R$. Also notice that they are the regularized time dependent counterparts of the  magnetic dipoles proposed in~\cite[Page 230]{ColtonKress} for the frequency domain. Moreover, they can be written as the convolution in time of the modulation function $\chi$ and  the fundamental solution of the wave equation as follows:
\begin{equation}
  \mathcal{E}^i(\x,t;\y,\p)
  = \chi \ast \curl _{\x} (\p\,  \Phi(\x-\y,\cdot)) 
  = \int_{\R} \chi(t-\tau)\, \curl_{\x} \left(\p \,\Phi(\x-\y,\tau)\right) d{\tau}\, .
\label{Ei_TD}   
\end{equation}
Let us recall that the fundamental solution of the wave equation satisfies, in the distributional sense,
\[
  c_0^{-2} \Phi_{tt} (\x,t)
  - \Delta_{\x} \Phi (\x,t) \, = \, \delta(|\x|) \,\delta(t)
  \qquad \text{for } (\x ,t)\in \R^3 \times \R \, .
\]
Thus, since the regularized dipole is  divergence free away from the source point $\y$, we have in the distributional sense that
\begin{equation}
  \label{eq:incPointSource}
  c_0^{-2}  \mathcal{E}^i_{tt} (\x,t;\y,\p)
  + \curl_{\x} \curl_{\x} ( \mathcal{E}^i(\x,t;\y,\p) )
  = \chi(t) \, \curl_{\x} \left(\p\, \delta(|\x-\y|) \right) \qquad \mbox{for  } (\x ,t) \in\R^3 \times \R \, .
\end{equation}

Let $\mathcal{E}(\x,t;\y,\p)$ denote  the scattered field corresponding to the incident field $\mathcal{E}^i(\x,t;\y,\p)$. The linearity of Maxwell's equations~(\ref{eq:TDfwd_Omegac}-\ref{eq:TDfwd_Gamma}) shows that the scattered field for a superposition of incident fields equals the superposition of the corresponding scattered fields. More generally, for a function $\f: \GammaI \times \R \to \R^3$ with $\f\in L^2(\mathbb{R}, \bfL_T^2(\GammaI))$, we may consider  the incident field {that is the superposition of fields due to dipoles whose polarizations are given by $\f$:}
\[
  (\Mop^i_\chi \mathbf{f} )(\x,t)
  \, =\, \int_{\R} \int_{\GammaI} 
    \mathcal{E}^i(\x,t-\tau;\y,\mathbf{f}(\y,\tau)) \, dA_{\y} \, d{\tau} \qquad \mbox{for } (\x,t) \in (\R^3\setminus\GammaI) \times \R \,.
\]
{Notice that this field is the counterpart of the usual Herglotz wave function in the acoustics setting. We also consider} the corresponding {generalized} scattered field {formed by a
weighted superposition of scattered fields due to dipoles}
$$
({\Mop}_\chi \mathbf{f} )(\x,t)\, =\, \int_{\R} \int_{\GammaI}
\mathcal{E}(\x,t-\tau;\y,\mathbf{f}(\y,\tau)) \, dA_{\y} \, d{\tau} \qquad \mbox{for } (\x,t) \in \Omega^c \times \R \, .
$$
In the following, we will make use of polarizations given by tangential fields on $\GammaI\times\R$, and then measure the tangential component of the scattered field on $\GammaM\times\R$. Accordingly,  we define the \emph{near-field operator}
applied to a vector function $\f\in L^2_{s_0}(\R;\bfL^2(\GammaI))$ by
\begin{equation}
 (\mathcal{N}_\chi \mathbf{f})(\x,t) \, =  \int_{\R} \int_{\GammaI}
    \mathcal{E}_T(\x,t-\tau;\y,\mathbf{f}(\y,\tau)) \, dA_{\y} \, d\tau
   \qquad \mbox{for } (\x,t) \in \GammaM \times \R \, ,
\label{NFO}
\end{equation}
where, as usual, the subscript $T$  refers to the tangential trace here taken on $\GammaM$ (i.e. $\mathbf{v} _T= (\boldsymbol{\nu} \times \mathbf{v} )\times \boldsymbol{\nu}$).

Concerning causality, we emphasize that even for a causal field $\mathbf{f}\in L^2(\R_+;\bfL^2_T(\GammaI))$, the corresponding incident field ${\Mop}^i_{\chi}\f$ is not necessarily causal (and hence, nor is the scattered field ${\Mop}_{\chi}\f$). However, the measured data $(\x,t) \mapsto \mathcal{E}_T(\x,t;\y,\p) $ that represents the kernel of the integral operator $\mathcal{N}_\chi$, are tangential components of causal electromagnetic waves. 

For later use, we note that the incident field operator ${\Mop}^i_\chi$ can be represented as the convolution in time of the modulation function $\chi$ with the vector potential defined by the non-regularized  magnetic dipole operator. Indeed, for tangential densities $\mathbf{f}\in L^2(\R;\bfL^2_T(\GammaI))$,  the non-regularized counterpart of ${\Mop}^i_{\chi}$ is 
\[
  ({\Mop}^i \mathbf{f})(\x,t) = \curl _{\x} \Big( \int_{\R} \int_{\GammaI}
    \Phi(\x-\y,t-\tau)\, \mathbf{f}(\y,\tau) \, dA_{\y} \, d\tau \Big) \qquad \mbox{for } (\x,t) \in (\R^3\setminus\GammaI) \times \R \, ,
\]
which is the curl of the (non-regularized) retarded single layer potential for the wave equation defined over the surface $\GammaI$; see the paragraph \ref{par:SLwave} for more details about this integral operator. Then
\begin{eqnarray*}
  ({\Mop}^i_\chi \mathbf{f})(\x,t)
 &=&  \displaystyle\int_{\R} \chi(t-\tau) \, \curl_{\x} ( \int_{\R} \int_{\GammaI}
      \Phi(\x-\y,\tau-\tau_1) \mathbf{f}(\y,\tau_1) \, dA_{\y} \, d{\tau_1}  ) \, d {\tau} \\
  &=&  (\chi \ast {\Mop}^i \mathbf{f}(\x,\cdot))(t) \qquad\text{for } (\x,t)\in ( \R^3\setminus\GammaI )\times\R \, .
\end{eqnarray*}

The Time Domain Linear Sampling Method (TD-LSM) is an imaging technique that yields a picture of the scatterer by approximately solving, for each sampling point, a linear integral equation {involving the near-field operator and }whose right-hand side is the tangential trace of a point source placed at the point under study. Using the measured scattered field, we can
compute the near-field operator (\ref{NFO}) applied to a vector function $\f\in L^2_{s_0}(\R;\bfL^2(\GammaI))$.
Then, for each sampling point $\z$, polarization $\p_{\z} \in \R^3\setminus\{ \bf0 \}$ and delay $\tau_{\z} \in \R$,  
we seek an approximate solution $\mathbf{g} = \mathbf{g}(\cdot,\cdot;\z,\p_{\z},\tau_{\z}) \in L^2(\R;\bfL^2_T(\GammaI))$ of the \emph{near-field equation}
\begin{equation}
  \label{eq:nearField}
  \big( \mathcal{N}_\chi \mathbf{g}(\cdot,\cdot;\z,\p_{\z},\tau_{\z}) \big)(\x,t)
    =  \mathcal{E}^i_T(\x,t-\tau_{\z};\z,\p_{\z})
    \qquad \text{for } (\x,t) \in  \GammaM \times \R \, .
\end{equation}
This is an ill-posed linear integral equation, {and an approximate solution can  
be obtained by Tikhonov regularization. This problem is the near-field, time domain analogue of  \cite[Eqn. (7.45)]{ColtonKress}.  We shall show that by solving (\ref{eq:nearField})
for $\z$ in a domain a priori known to contain the scatterer, and then using an indicator function
based on norms of this solution, we can obtain an approximation to the shape of the scatterer.  For more details of the LSM approach see Section~\ref{sec:num} and more generally \cite{ccm-book}. }

In our numerical tests,  $\tau_{\z}$ is kept constant for all sampling points $\z$ in a test region that we choose {a priori} to search for the scatterer: Theoretically, one might also let $\tau_{\z}$ vary depending on $\z$, but the assumption that this is not the case allows to   neglect the dependence of $\mathbf{g}$ on  $\tau_{\z}$. Moreover, $\p_{\z}=\p$ is typically fixed to be a unit vector ($|\p |=1$). Note that in their analysis of the TD-LSM for the wave equation \cite{prunty19}, the authors argue that one may choose $\tau_{\z}=\tau=0$.

By solving the near-field equation approximately for many sampling points $\z$ we construct an indicator function for the scatterer.  Details of this
procedure are given in Remark~\ref{ind_rem}.

\subsection{Basic ingredients for the TD-LSM analysis}\label{subsec:tools4LSManalysis}

We next study some basic tools for the theoretical justification of the TD-LSM. More precisely, we start by recalling an integral operator related to the wave equation: the so-called retarded single layer potential. Then we study the operator ${\Mop}^i_{\chi}$ related to the superposition of incident fields, and the impedance trace operator that maps incident fields onto the associated boundary data in the impedance boundary condition (\ref{eq:TDfwd_Gamma}) of the forward problem. These results lay the foundation to deduce some basic properties of the near-field measurement operator $\mathcal{N}_{\chi}$.

\subsubsection{Retarded single layer potential for the wave equation}\label{par:SLwave}
The regularized  retarded single layer potential for the wave equation  defined over $\GammaI$ is
\begin{equation}\label{eq:regsingleLayer}
    \mathcal{S}_{\GammaI ,\chi}  f (\x,t)= 
    \int_{\R} \int_{\GammaI} f(\y,\tau) \Phi_{\chi}(\x-\y,t-\tau) \, dA_{\y} \, d\tau \qquad\text{for } (\x,t)\in\R^3\times\R \, .
  \end{equation}
Notice that it is the regularization by means of the modulation function $\chi$ of the non-regularized retarded single layer potential for the wave equation defined over $\GammaI$:
  \begin{equation}
    \label{eq:singleLayer}
    \mathcal{S}_{\GammaI} f (\x,t) =
    \int_{\R} \int_{\GammaI} f(\y,\tau) \Phi (\x-\y,t-\tau) \, dA_{\y} \, d\tau \qquad\text{for } (\x,t)\in\R^3\times\R\, ,
  \end{equation}
   cf. \cite{Sayas2016}. 
  The latter defines a bounded operator from $L^2_{s_0}(\R;L^2(\GammaI))$
 into $L^2_{s_0}(\R;L^2_{loc}(\R^3))$ that preserves causality. Moreover, the following result provides bounds on the single layer operator.

\begin{lemma}\label{lem:trSLGammaI}
Let  $s_0>0$. Then $\mathcal{S}_{\GammaI}$ is a bounded  operator from $L^2(\R;L^2(\GammaI))$ into $H^{-1}_{s_0}(\R;H^1(\Omega^c))$, and also into $H^{-1}_{s_0}(\R;H^1(\Omega))$. Moreover, its trace on $\Gamma\times\R$ (that is,  $\mathcal{S}_{\GammaI}|_{\Gamma\times\R}$)  is continuous across $\Gamma\times\R$.
\end{lemma}
\begin{remark}
A stronger result can be proved for the regularized single layer~\cite{MarmoratEtal}.
\end{remark}
\begin{proof} The Fourier-Laplace transform $\hat{w}_s$
  of $w=\mathcal{S}_{\GammaI}f$ satisfies a transmission problem in $\R^3\setminus\GammaI$; more precisely, $\Delta\hat{w}_s+k_s^2\hat{w}_s=0$ in $\R^3\setminus\GammaI$, and its trace is continuous across $\GammaI$ whereas
  its normal derivative has a jump  equal to the  density $f$; accordingly, a variational formulation of this transmission problem leads to
  \[
    \int_{\R^3} \Big( |\nabla \hat{w}_s|^2
      - \frac{s^2}{c_0^2}\, |\hat{w}_s|^2 \Big) \, d{\x}
    = \int_{\GammaI} \hat{f}_s \overline{\hat{w}_s} \, dA \, .
  \]
 Taking the imaginary part of the product by $-\ol{s}$, and using Cauchy's generalized inequality and the trace theorem, 
  \[
    \Im (s) \int_{\R^3} \Big( |\nabla \hat{w}_s|^2
      + \frac{|s|^2}{c_0^2} |\hat{w}_s|^2 \Big) \, d{\x}
     \leq |s| \, \| \hat{f}_s \|_{H^{-1/2}(\GammaI)} \| \hat{w}_s \|_{H^{1/2}(\GammaI)}
     \leq \frac{|s|^2 }{4 \alpha}\, \| \hat{f}_s \|^2_{H^{-1/2}(\GammaI)}
      + C_{\GammaI} \alpha\, \| \hat{w}_s \|^2_{H^1(\R^3)} \, ,
  \]
  for any $\alpha>0$. In particular, if $\alpha >0$ is small enough, we deduce that
  \begin{equation}\label{eq:auxSL_FL}
   \| \hat{w}_s \|_{H^1(\R^3)}
    \leq C |s|\, \| \hat{f}_s \|_{H^{-1/2}(\GammaI)} \, ,
  \end{equation}
  where $C>0$ does not depend on $s \in \C_{s_0}$. Now we can infer information back to the time domain using Lemma \ref{lem:lubich}, and it follows that, for any $p\in\mathbb{R}$, the single layer potential $\mathcal{S}_{\GammaI}$ is
  bounded as a map from $H^p_{s_0}(\R;H^{-1/2}(\GammaI))$ into
  $H^{p-1}_{s_0}(\R;H^1(\R^3))$ for any $p\in\R$; in particular, it is bounded from $L^2_{s_0}(\R;L^2(\GammaI))$ into
  $H^{-1}_{s_0}(\R;H^1(\Omega))$ and into $H^{-1}_{s_0}(\R;H^1(\Omega^c))$.
  \end{proof}

\subsubsection{Analysis of the operator associated to the superposition of incident fields}\label{sec:opManalysis}

The aim of this paragraph is to study the operator ${\Mop}^i_\chi$ defined in~\eqref{eq:incPointSource} between suitable Sobolev spaces. To accomplish this, we define the following closed subspace of $H^p_{s_0}(\R;\bfL^2(\Omega))$ associated to incident fields:
\begin{equation}\label{eq:defZOmega}
  Z^p_{s_0}(\Omega) = \left\{ \g \in H^p_{s_0}(\R;\bfL^2(\Omega)) \, \big{|}\,
      \Ddiv \g = 0 \text{ and } \g_{tt} + c_0^2 \curl (\curl \g) = \bf0 \text{ in } \Omega\times\R  \right\} .
\end{equation}
Then we have the following result:
\begin{theorem}
  \label{th:boundM}
  Assume that $\chi \in \mathcal{C}^\infty_0(\R_+)$.
  Then ${\Mop}^i_\chi$ is a bounded and injective operator from $L^2_{s_0}(\R;\bfL^2_T(\GammaI))$ into
  $Z^p_{s_0}(\Omega)$ for any $p \in \R$, and the range of its tangential traces $\{ ({\Mop}^i_\chi\f )_T \,|\,\f\in L^2_{s_0}(\R;\bfL^2_T(\GammaI))\}$ is dense in $H^p_{s_0}(\R;\bfL^2_T(\Gamma))$.
\end{theorem}
\begin{proof}
  We first study the operator ${\Mop}^i_\chi$ based on explicit bounds in $s\in\C_{s_0}$ of its Fourier-Laplace transform\
  \[
    (\widehat{{\Mop}}^i_{\chi,s} \hat{\f}_s)(\x)
    = \hat{\chi}_s \, \curl_{\x}\Big(
      \int_{\GammaI} \hat{\f}_s(\y) \frac{e^{ic_0^{-1}s |\x-\y|}}{4 \pi \, |\x-\y|} \, dA_{\y}\Big)
      \qquad\mbox{for } \x \in \Omega \, . 
  \]
  Remaining in the Fourier-Laplace domain, we note that the single layer potential for vector fields is defined to act componentwise. Then, the mapping 
  \[
    \hat{\f}_s \mapsto \curl (\widehat{\mathcal{S}}_{\GammaI,s}\hat{\f}_s) 
  \]
  is bounded from $\bfL^2(\GammaI)$ into $\bfL^2(\R^3)$ because from (\ref{eq:auxSL_FL}) we have that
  $$\| \curl (\widehat{\mathcal{S}}_{\GammaI,s}\hat{\f}_s)  \|_{\bfL^2(\R^3)} \leq \sqrt{3}\, \| \widehat{\mathcal{S}}_{\GammaI,s}\hat{\f}_s \|_{\bfH^1(\R^3)}\leq  
    C |s| \,\| \hat{\f}_s \|_{\bfL^2(\GammaI)} \, ,
  $$ 
  where $C>0$ is independent of $s\in\C_{s_0}$.
  Also notice that the compactness of the support of the modulation function $\chi \in \mathcal{C}^\infty_0(\R_+)$ guarantees that its Fourier-Laplace transform
  $\hat{\chi}_s$ decays faster than any algebraic rate
  (see~\cite{Rudin1973}), that is,
  $|\hat{\chi}_s | \leq C_{\chi,p}\, |s|^{-p}$ for $s \in\C_{s_0}$
  and $p >0$. This shows that, for any  $s\in\C_{s_0}$ and $p >0$, we can bound $\widehat{{\Mop}}^i_{\chi,s} \hat{\f}_s=\hat{\chi}_s  \curl (\widehat{\mathcal{S}}_{\GammaI,s}\hat{\f}_s)$ by
  \[
    \| \widehat{{\Mop}}^i_{\chi,s} \hat{\f} \|_{\bfL^2(\R^3)}
    \leq C_{\chi,p}\, |s|^{1-p} \| \hat{\f}_s \|_{\bfL^2(\GammaI)} \, ,
  \]
  where $C_{\chi,p}>0$ does not depend on $s\in\C_{s_0}$. 
  Back to the time domain, it follows that ${\Mop}^i_\chi$ is bounded from $H^{q}_{s_0}(\R;\bfL^2(\GammaI))$
  into $H^{q-1+p}_{s_0}(\R; \bfL^2(\R^3))$ for every $p,q\in\R$. In particular, it is bounded from $L^{2}_{s_0}(\R;\bfL^2_T(\GammaI))$
  into $H^{p}_{s_0}(\R; \bfL^2(\R^3))$ for any $p\in\R$.
  Moreover, the divergence of ${\Mop}^i_\chi \f$ vanishes in $\R^3\times\R$ from its definition, and it is also clear that
  ${\Mop}^i_\chi \f$ is a weak solution to the vector wave equation
  in $ (\R^3\setminus\GammaI) \times \R$. Hence, ${\Mop}^i_\chi$ is also
  bounded from $L^{2}_{s_0}(\R;\bfL^2_T(\GammaI))$
  into $Z^q_{s_0}(\OmegaI)$ for any $p\in\R $, where  the space $Z^q_{s_0}(\OmegaI)$ is defined as in (\ref{eq:defZOmega}).  

  Next we show that ${\Mop}^i_\chi$ is injective. To this end, let us consider   $\f \in H^p_{s_0}(\R;\bfL^2_T(\GammaI))$ such that   ${\Mop}^i_\chi \f$ vanishes in $Z^p_{s_0}(\Omega)$.
  Then $\widehat{{\Mop}}^i_{\chi,s} \hat{\f}_s$ vanishes for almost all complex frequencies $s \in\C_{s_0}$; equivalently, since  $\hat{\chi}_s$ is an entire function, the field
  \[
    \hat{\w}_s (\x) = \curl_{\x}
    \Big(\int_{\GammaI} \hat{\f}_s(\y) \frac{e^{is c_0^{-1} |\x-\y|}}{4 \pi |\x-\y|} \, d A_{\y}\Big) =\curl \hat{\mathcal{S}}_{\GammaI,s}\hat{\f}_s (\x) = \bf0 
  \]
  in $\bfL^2(\Omega)$ for almost every $s\in\C_{s_0}$.
  Moreover, $ \hat{\w}_s\in\bfH^1(\R^3\setminus\GammaI)$ solves the vector Helmholtz equation
  $\Delta \hat{\w}_s+\frac{s^2}{c_0^2} \hat{\w}_s=\bf0$   in $\R^3 \setminus \GammaI$. Thus, the fact that it vanishes in $\Omega$ implies that it also vanishes up to $\GammaI$; notice that, in case $\GammaI$ is open, by analytic continuation $\hat{\w}_s$ vanishes up to the whole $\tilde{\Gamma}_{\mathrm{I}}$ and therefore by this reasoning we may simplify the situation by identifying $\tilde{\Gamma}_{\mathrm{I}}$ with $\GammaI$. 
  Let us recall that $\curl\hat{\w}_s\times\boldsymbol{\nu}$ is continuous across $\GammaI$, cf. \cite[Theorem 6.12]{ColtonKress}\footnote{Th. 6.12 in \cite{ColtonKress} states this continuity result for a smooth surface $\GammaI$ and a density field
  $\hat{\f}_s\in\boldsymbol{\mathcal{C}}^1(\GammaI)$, but the result also holds for a $\mathcal{C}^2$ surface $\GammaI$ and a field $\hat{\f}_s\in\bfL^2(\GammaI)$.}. Hence, if we consider the  field $\hat{\w}_s$ outside of $\GammaI$, it follows that it satisfies the boundary condition $\curl\hat{\w }_s \times\bfnu =\bf0$ on $\GammaI$. Furthermore, $\hat{\w}_s$ is smooth and  solves the Maxwell's equation $\curl\curl \hat{\w }_s- k_s^2 \hat{\bfw }_s=\bf0 $ outside of $\GammaI$, that is, in {$\OmegaI^{c}=\mathbb{R}^3\setminus\overline{\OmegaI}$}. 
  By the uniqueness of solutions to this exterior problem (which can be shown reasoning as we did above Th. \ref{th:FwdProbOK} for the exterior impedance problem  in the Fourier-Laplace domain), we deduce that  $\hat{\w}_s$ also vanishes outside of $\GammaI$. Finally, notice that the jump of $\hat{\w}_s$ across $\GammaI$ is the following,
  see~\cite[Theorem 6.12]{ColtonKress}:
  \[
    \lim_{h\to 0} \left( \hat{\w}_s(\x + h \boldsymbol{\nu})
       - \hat{\w}_s(\x - h \boldsymbol{\nu}) \right)
    = \hat{\f}_s \times\boldsymbol{\nu} \qquad\text{for } \x\in\GammaI \, .
  \]
  Therefore, from $\hat{\w}_s=\bf0$ in $\R^3\setminus\GammaI$ it follows that the tangential field $\hat{\f}_s= \hat{\f}_{s,T}= \boldsymbol{\nu} \times (\hat{\f}_s\times \boldsymbol{\nu} ) $ vanishes on $\GammaI$ for almost every $s\in\C_{s_0}$. Back to the time domain, we conclude that also $\f$ vanishes in $\GammaI\times\R$.

  It remains to show that the tangential trace of ${\Mop}^i_\chi$ has dense range in $L^2_{s_0}(\R;\bfL^2_T(\Gamma ))$. To this end, we study the injectivity of the adjoint of
  \[
    ({\Mop}^i_{\chi,T} \f)(\x,t) = 
       \Big(  \int_{\R}  \int_\R \int_{\GammaI} \chi (t-\tau_1) \,
       \curl_{\x} (\f(\y,\tau)\Phi(\x-\y,\tau_1-\tau)) \, dA_{\y} \, d{\tau} \, d\tau_1 \Big) _{\! T} \quad \mbox{for }(\x,t)\in\Gamma\times\R 
  \]
  with respect to the inner product of $L^2_{s_0}(\R; \bfL^2_T(\Gamma ))$. Notice that, for $\g\in L^2_{s_0}(\R;\bfL^2_T(\Gamma ))$ and $\f\in L^2_{s_0}(\R;\bfL^2_T(\GammaI ))$, we formally have that
   \[
   \begin{array}{l}
    \langle  {\Mop}_{\chi,T}^{i,\ast} \g , \f \rangle _{ L^2(\R;\bfL^2_T(\GammaI))} = 
     \langle  \g ,  {\Mop}_{\chi,T}^{i}\f \rangle _{ L^2(\R;\bfL^2_T(\Gamma))} = 
     \displaystyle\int_{\Gamma} \int_\R \g(\x,t)\cdot \overline{ ({\Mop}^i_{\chi,T} \f)(\x,t)}\, dt \, dA_{\x} \\
     \displaystyle = 
     \displaystyle\int_{\Gamma} \int_\R \g(\x,t)\cdot \overline{ \left(\int_{\R} \int_\R \int_{\GammaI} \chi (t-\tau_1) \,
       \curl_{\x} (\f(\y,\tau)\Phi(\x-\y,\tau_1-\tau)) \, dA_{\y} \, d{\tau} \, d\tau_1\right) }\, dt  \, dA_{\x} \, .
      \end{array}
  \] 
Besides, at almost every $t,\tau,\tau_1\in\R,\, \x\in \Gamma, \y \in\GammaI$ it holds
    \[
       \begin{array}{l}
       \g(\x,t)\cdot \curl_{\x} (\overline{ \f(\y,\tau)\Phi(\x-\y,\tau_1-\tau) }) = - \g(\x,t)\cdot (\overline{  \f(\y,\tau)}\times\nabla_{\y}\Phi(\x-\y,\tau_1-\tau)) \\[1ex]
\hspace*{1cm}       =  \overline{  \f(\y,\tau)} \cdot  \curl_{\y} ( \g(\x,t) \Phi(\x-\y,\tau_1-\tau))
\, .
             \end{array}
  \] 
Hence
  \[
   \begin{array}{l}
    \langle  {\Mop}_{\chi,T}^{i,\ast} \g , \f \rangle _{ L^2(\R;\bfL^2_T(\GammaI))} =      \displaystyle 
     \displaystyle\int_{\GammaI} \int_\R  \overline{  \f(\y,\tau)} \cdot \left( \int_{\R} \int_\R \int_{\Gamma} \chi (t-\tau_1) \curl_{\y} ( \g(\x,t) \Phi(\x-\y,\tau_1-\tau))   \, dA_{\x} \,  dt \, d{\tau_1} \right)\, d\tau\, dA_{\y}   \, ,
      \end{array}
  \] 
so that
  \[
    ({\Mop}_{\chi,T}^{i,\ast} \g)(\y,\tau)  = 
    \int_\R \int_\R \int_{\Gamma} \chi (t-\tau_1)
      \curl_{\y} ( \g(\x,t)\,\Phi(\x-\y,\tau_1-\tau) )\, dA_{\x}\,  dt \, d{\tau_1} 
      \, .
  \]
 This may be rewritten as
 \[
  \begin{array}{l}
    ({\Mop}_{\chi,T}^{i,\ast} \g)(\y,\tau) 
 = 
\displaystyle  \curl_{\y} \Big(   \int_\R \int_\R \int_{\Gamma} \chi (-\tau-\tau_1)
      \g(\x,-t)\,\Phi(\x-\y,\tau_1-t) \, dA_{\x} \, dt \, d{\tau_1} \Big) 
       = 
({\Mop}_{\chi,\Gamma,T}^{i} \tilde{\g})(\y,-\tau) %
\, ,
      \end{array}
  \]
where $\tilde{\g}(\y,\tau)=\g(\y,-\tau)$ and ${\Mop}_{\chi,\Gamma,T}^{i} \g$ denotes the tangential trace on $\GammaI\times\R $ of  
\[
  ({\Mop}^i_{\chi,\Gamma} \g)(\x,t)
  = \int_{\R} \int_{\Gamma}
    \mathcal{E}^i(\x,t-\tau;\y,\g (\y,\tau)) \, dA_{\y } \, d{\tau} \qquad \text{for }\g \in L^2_{s_0}(\R;\bfL^2_T(\Gamma))\, .
\]
The integral operator ${\Mop}_{\chi,\Gamma}^{i}$ represents the linear combination of incident fields taking now the sources over $\Gamma$; accordingly, the injectivity of the tangential trace of ${\Mop}_{\chi,\Gamma}^{i}$  on $\GammaI\times\R$ can be shown in a similar way as that of ${\Mop}_{\chi}^{i}$ on $\Gamma\times\R$. Thus, we conclude that ${\Mop}_{\chi,T}^{i,\ast}$ is injective in $L^2_{s_0}(\R;\bfL^2_T(\Gamma ))$.
\end{proof}

\subsection{Analysis of the impedance trace operator}\label{sec:impedanceTr}

We next study the operator that relates incident fields in $Z^p_{s_0}(\Omega)$ with data in the impedance boundary condition (\ref{eq:TDfwd_Gamma}). More precisely, for any admissible incident field $\g\in Z^p_{s_0}(\Omega)$, the right-hand side of (\ref{eq:TDfwd_Gamma}) is given by $\boldsymbol{\mathcal{G}}_{\g , t}=c_0\curl\g\times\boldsymbol{\nu}-\Lambda \g_{t,T}$, where
$\boldsymbol{\mathcal{G}}_{\g}=-\boldsymbol{\mathcal{H}}_{\g}\times\boldsymbol{\nu}-\Lambda \g_T$ in $\R\times\Gamma$ and $\boldsymbol{\mathcal{H}}_{\g}$ is a causal solution of   $\boldsymbol{\mathcal{H}}_{\g,t}=-c_0\curl\g$ in $\Omega\times\R$. This leads us to define the impedance trace operator $\mathcal{T}_{\Gamma}$ by $\mathcal{T}_{\Gamma}\g=c_0\curl\g\times\boldsymbol{\nu}-\Lambda \g_{t,T}$ for fields $\g\in Z^p_{s_0}(\Omega)$. Note that any field $\g\in Z^p_{s_0}(\Omega)$ belongs to $ H^p_{s_0}(\R;\bfL^2(\Omega)) $, is divergence free and satisfies $\g_{tt} + c_0^2 \curl (\curl \g) = \bf0$ in $\Omega\times\R $; in particular $\g\in  H^{p-2}_{s_0}(\R;\bfH(\curl\curl,\Omega))$, so that the impedance trace operator $\mathcal{T}_{\Gamma}: Z^p_{s_0}(\Omega) \to H^{p-2}_{s_0}(\R;\bfL^2_T(\Gamma)) $ is well-defined and bounded. We need to further understand if $\mathcal{T}_{\Gamma}$ is injective and if it has dense range in $H^{p-2}_{s_0}(\R;\bfL^2_T(\Gamma))$, and we will do so by moving to the Fourier-Laplace domain. Notice that, formally, the Fourier-Laplace transform of  $\mathcal{T}_{\Gamma}\g$ is  $\hat{\mathcal{T}}_{\Gamma ,s}\hat{\g}_s=c_0\curl\hat{\g}_s\times\boldsymbol{\nu}+is\Lambda \hat{\g}_{s,T}$. 

In order to study the injectivity of $\mathcal{T}_{\Gamma}$, the natural approach is to consider some $\g\in Z^p_{s_0}(\Omega)$ such that $c_0\curl\g\times\boldsymbol{\nu}-\Lambda \g_{t,T}=\bf0$ on $\R\times\Gamma$ and handle this problem as we did for the forward problem; 
however, doing so we cannot deduce coercivity properties because of the signs that we get when working in the interior region $\Omega$. 
To overcome this difficulty, we notice that to analyze the near-field operator we can restrict the impedance trace operator to act on fields in the range of ${\Mop}^i_{\chi}$. Hence, it is enough to study the impedance trace as an operator from ${\Mop}^i _{\chi}(L^2_{s_0}(\R;\bfL^2_T(\GammaI)))\subseteq Z^p_{s_0}(\OmegaI)$ into $H^{p-1}_{s_0}(\R;\bfL^2_T(\Gamma ))$. Accordingly, we take a field $\g\in Z_{s_0}^p(\OmegaI)$ such that
$$
c_0\curl\g\times\boldsymbol{\nu}-\Lambda \g_{t,T}=\bf0 \text{ on } \Gamma \times\R \, .
$$
Then for almost all $s\in\C_{s_0}$, its Fourier-Laplace transform $\hat{\g}_s$ is divergence free and satisfies $-s^2 \hat{\g}_s + c_0^2 \curl \curl \hat{\g}_s = \bf0$ in $\OmegaI$, and the condition on the impedance trace translates into
$c_0\curl\hat{\g}_s\times\boldsymbol{\nu}+is\Lambda \hat{\g}_{s,T}=\bf0$ on $\Gamma$. Integrating by parts once in ${\OmegaI^{aux}}=\Omega^c\cap \OmegaI$, we have
\begin{equation}\label{eq:intxpartsOmegaIc}
\| \curl\hat{\bfg}_s\|^2_{\bfL^2({\OmegaI^{aux}})}  - k_s^2\, \| \hat{\bfg}_{s} \|^2_{\bfL^2({\OmegaI^{aux}})}  - ik_s \int_{\Gamma}\Lambda\hat{\bfg}_{s,T}  \cdot \overline{\hat{\bfg}}_{s,T} \, dA +  \int_{\GammaI}\boldsymbol{\nu}\times\curl\hat{\bfg}_{s}  \cdot \overline{\hat{\bfg}}_{s,T} \, dA=0 \, .
\end{equation}
Multiplying both hands of (\ref{eq:intxpartsOmegaIc}) by $\overline{s}$ and identifying the imaginary parts,  
\begin{equation*}
\begin{array}{l}
\displaystyle - \Im (s) \left(
\| \curl\hat{\bfg}_s\|^2_{\bfL^2({\OmegaI^{aux}})}  + |k_s|^2\, \| \hat{\bfg}_{s} \|^2_{\bfL^2({\OmegaI^{aux}})}  \right) 
+ c_0\, |k_s| \int_{\Gamma}\Lambda\hat{\bfg}_{s,T}  \cdot \overline{\hat{\bfg}}_{s,T} \, dA\\
\qquad + \, \displaystyle\Im \left(\overline{s} \int_{\GammaI}\boldsymbol{\nu}\times\curl\hat{\bfg}_{s}  \cdot \overline{\hat{\bfg}}_{s,T} \, dA \right) =0 \, .
\end{array}
\end{equation*}
Similarly, we integrate by parts in $\OmegaI$, multiply both hands by $\overline{s}$ and identify the imaginary parts, so that
\begin{equation*}
\displaystyle - \Im (s) \left(
\| \curl\hat{\bfg}_s\|^2_{\bfL^2(\OmegaI)}  + |k_s|^2\, \| \hat{\bfg}_{s} \|^2_{\bfL^2(\OmegaI)}  \right) 
+ \Im \left(\overline{s} \int_{\GammaI}\boldsymbol{\nu}\times\curl\hat{\bfg}_{s}  \cdot \overline{\hat{\bfg}}_{s,T} \, dA \right) =0 \, .
\end{equation*}
Hence, 
\begin{equation*}
\begin{array}{l}
\displaystyle 0\leq c_0\, |k_s| \int_{\Gamma}\Lambda\hat{\bfg}_{s,T}  \cdot \overline{\hat{\bfg}}_{s,T} \, dA = \\
\displaystyle \hspace*{1cm} = \Im (s) \, \big(\| \curl\hat{\bfg}_s\|^2_{\bfL^2({\OmegaI^{aux}} )}  + |k_s|^2\, \| \hat{\bfg}_{s} \|^2_{\bfL^2({\OmegaI^{aux}})} \big) - \Im \Big(\overline{s} \int_{\GammaI}\boldsymbol{\nu}\times\curl \hat{\bfg}_{s}  \cdot \overline{\hat{\bfg}}_{s,T} \, dA \Big) =\\
\displaystyle \hspace*{1cm} = \Im (s)\, \big(\| \curl\hat{\bfg}_s\|^2_{\bfL^2({\OmegaI^{aux}} )}  + |k_s|^2\, \| \hat{\bfg}_{s} \|^2_{\bfL^2({\OmegaI^{aux}})} \big) - \Im (s) \, \big(
\| \curl\hat{\bfg}_s\|^2_{\bfL^2(\OmegaI)}  + |k_s|^2\, \| \hat{\bfg}_{s} \|^2_{\bfL^2(\OmegaI)}  \big)  =\\[1ex]
\displaystyle \hspace*{1cm} = - \Im (s) \left(
\| \curl\hat{\bfg}_s\|^2_{\bfL^2(\Omega)}  + |k_s|^2\, \| \hat{\bfg}_{s} \|^2_{\bfL^2(\Omega )}  \right)\leq 0 \, .
\end{array}
\end{equation*}
Thus $\hat{\g}_s=\bf0$ in $\Omega$ and, by the unique continuation principle (notice that $ -\frac{s^2}{c_0^2} \hat{\g}_s -\Delta \hat{\g}_s = \bf0$ in $ \OmegaI $), also in $\Omega_{\mathrm{I}}$; moving back to the time domain, we conclude that $\g=\bf0$ in $\OmegaI\times \R$.

Let us next prove that the range of the impedance trace is dense in $H^{p-2}_{s_0}(\R;\bfL^2_T(\Gamma))$. Given any $\boldsymbol{\mathcal{G}}$ in $H^p_{s_0}(\mathbb{R};\bfL^2_T(\Gamma ))$, the well-posedness of the forward exterior problem in the time domain (see Lem. \ref{th:FwdProbOK}) allows us to consider $\boldsymbol{\mathcal{W}}\in H^p_{s_0}(\mathbb{R};\bfL^2(\Omega^c))\cap H^{p-1}_{s_0}(\mathbb{R};\bfH(\curl,\Omega^c))$ that solves
\begin{eqnarray*}
\boldsymbol{\mathcal{W}}_{tt}+c_0^2\curl \curl\boldsymbol{\mathcal{W}} \, =\, \bf0  \qquad & \mbox{in } \Omega^c \times \mathbb{R} \, , \\
(\curl\boldsymbol{\mathcal{W}} )\times\bfnu- c_0^{-1} \Lambda\boldsymbol{\mathcal{W}}_{t,T} \, =\, c_0 ^{-1}\boldsymbol{\calG } \qquad & \mbox{on }\Gamma \times \mathbb{R}\, .
\end{eqnarray*}
Notice that this field satisfies  ${\mathcal{T}}_{\Gamma}\boldsymbol{\mathcal{W}}=\boldsymbol{\mathcal{G}}$. Besides, Th. \ref{th:boundM} guarantees that there is a sequence $(\f_k )_{k\in\mathbb{N}}\subset L^2_{s_0}(\R;\bfL^2_T(\GammaI))$ such that $({\Mop}^i_{\chi}\f_k)_T\to\boldsymbol{\mathcal{W}}_T$ in $H^{p-1}_{s_0}(\R;\bfL^2_T(\Gamma))$. Recalling here again the well-posedness of the exterior problem, we have that $ {\Mop}_{\chi}\f_k \to \boldsymbol{\mathcal{W}} $ in $H^{p-2}_{s_0}(\mathbb{R};\bfH(\curl,\Omega^c)) $; in particular, thanks to the continuity of tangential traces, it follows that ${\mathcal{T}}_{\Gamma} {\Mop}_{\chi}\f_k 
\to {\mathcal{T}}_{\Gamma}\boldsymbol{\mathcal{W}}\, =\, \boldsymbol{\calG} $ in $H^{p-3}_{s_0}(\R;\bfL^2_T(\GammaI)) $ and we conclude that the range of ${\mathcal{T}}_{\Gamma}$ when applied to ${\Mop}^i_{\chi}(L^2_{s_0}(\R;\bfL^2_T(\GammaI)))$ is dense in $H^{p-3}_{s_0}(\mathbb{R};\bfL^2_T(\Gamma))$.

Summing up, we have shown the following result.
\begin{theorem}  \label{th:boundT}
For any $p\in\R$, the operator ${\mathcal{T}}_{\Gamma}: Z^p_{s_0}(\Omega) \to H^{p-2}_{s_0}(\R;\bfL^2_T(\Gamma)) $ is well-defined and bounded. Moreover, the composition ${\mathcal{T}}_{\Gamma}\circ{\Mop}^i_{\chi}: L^2_{s_0}(\R;\bfL^2_T(\GammaI)) \to H^{p}_{s_0}(\mathbb{R};\bfL^2_T(\Gamma))$ is injective and has dense range.
\end{theorem}

\subsection{Analysis of the operator associated to data in the LSM}\label{sec:LSManalysis}

In this paragraph we analyze the Linear Sampling Method (LSM) presented at the beginning of Section~\ref{sec:LSM}.
Recall that we have proposed to define the indicator function of the LSM based on solutions of the near-field equation
\begin{equation}
  \big( \mathcal{N}_\chi \mathbf{g}(\cdot,\cdot;\z,\p,\tau) \big)(\x,t)
    =\mathcal{E}^i_T(\x,t-\tau;\z,\p)
    \qquad \mbox{for }  (\x,t) \in  \GammaM \times \R \, ,\label{NFE_2}
\end{equation}
see \eqref{eq:nearField}. In the definition of the near-field operator for $\f\in L^2_{s_0}(\R;\bfL^2(\GammaI))$,
\[
 (\mathcal{N}_\chi \mathbf{f})(\x,t) \, =\, \int_{\R} \Big(\int_{\GammaI}
    \mathcal{E}_T(\x,t-\tau;\y,\mathbf{f}(\y,\tau))\, dA_{\y} \Big) \,  d\tau   \qquad \mbox{for } (\x,t) \in \GammaM \times \R \, ,
\]
we take the impedance trace of the incident field ${\Mop}^i_{\chi}\f$ and build the associated solution $\mathcal{Q} (\mathcal{T}_{\Gamma}{\Mop}^i_{\chi}\f)$ in $\Omega^c\times\R$, from which we measure its tangent trace on $\GammaM\times\R$. That is,
$$
\mathcal{N}_{\chi}\f = (\mathcal{Q} {\mathcal{T}}_{\Gamma} {\Mop}^i_{\chi}\f )_T \quad\text{on }\GammaM\times\R \, . 
$$
Notice that the tangential trace  is well-defined and onto when understood from $H^{p-2}_{s_0}(\R;\bfH(\curl,\Omega^c))$ into $H^{p-2}_{s_0}(\R;\bfH^{-1/2}_T(\curl_{\Gamma},\GammaM))$, where $\curl_{\Gamma}$ represents the surface curl on $\GammaM$.
Then the following exterior problem in $\R^3\setminus\ol{\OmegaM}$ with a tangential condition on $\GammaM$ is well-posed and preserves causality:
\begin{eqnarray*}
\bfu_{tt}+\frac{c_0}{\mu_0}\curl\curl\bfu = \bf0&\qquad \text{in }(\R^3\setminus\ol{\OmegaM})\times\R \, ,\\
\bfu_T = \bfg &\text{on }\GammaM\times\R\, ,
\end{eqnarray*}
where we seek $\bfu\in H^{p-2}_{s_0}(\R;\bfH (\curl,\R^3\setminus\ol{\OmegaM})) $ for any given $\bfg\in H^{p-2}_{s_0}(\R; \bfH^{-1/2}_T(\curl_{\Gamma},\GammaM))$. Hence, by Th. \ref{th:boundT} and Lem. \ref{th:FwdProbOK} it is straightforward that the near-field operator satisfies the following property.
\begin{proposition}\label{prop:N1-1denserange}
The near-field operator 
$\mathcal{N}_{\chi}$ is well-defined from $L^2_{s_0}(\R;\bfL^2_T(\GammaI))$ into $H^{p}_{s_0}(\R;\bfH^{-1/2}_T (\GammaM))$ for any $p\in\R$, and it is one-to-one with dense range. 
\end{proposition}

The following result, which is the main theoretical contribution of our paper, gives a partial justification for the use of the near-field equation to solve the inverse scattering problem under the assumption 
that the target is in $\OmegaM$ (which is assumed to satisfy that $\OmegaM\subseteq\OmegaI$, i.e. the receivers are not further from the target than the sources).

\begin{theorem}\label{main_one} {Assume that $\Omega\subset \OmegaM\subseteq\OmegaI$.}
  Let $s_0>0$, $\tau\in\R$ and $\p\in\R^3$ with $|\p |=1$.
  \begin{description}
      \item[Case 1.] Let $\z\in\Omega$. There is a sequence $\g_n\equiv \g_{n}(\cdot,\cdot;\z,\p) \in L^2_{s_0}(\R;\bfL^2_{T}(\GammaI))$ such that $\mathcal{N}_{\chi}\g_n\to \mathcal{E}^i_T(\cdot,\cdot-\tau;\z,\p)$ in $H^{-1}_{s_0}(\R,\bfH^{-1/2}_T(\GammaM))$  when $n\to\infty $ and such that ${\Mop}^i_{\chi}\g_n\in Z^p_{s_0}(\Omega)$ is bounded. 
      \item[Case 2.]  Let $\z\in{\OmegaM^{aux}=\Omega^c\cap \OmegaM}$, and consider any sequence $\g_n\equiv \g_{n}(\cdot,\cdot;\z,\p) \in L^2_{s_0}(\R;\bfL^2_{T}(\GammaI))$ such that $\mathcal{N}_{\chi}\g_n\to \mathcal{E}^i_T(\cdot,\cdot-\tau;\z,\p)$ in $H^{p}_{s_0}(\R,\bfH^{-1/2}_T(\GammaM))$  when $n\to\infty $. Then such sequence cannot be bounded in $L^2_{s_0}(\R;\bfL^2_{T}(\GammaI))$.
  \end{description}
  These sequences can be built by solving (\ref{eq:nearField}) with a Tikhonov regularization, that is, by minimizing
  $$
  \| \mathcal{N}_\chi \mathbf{g}_n(\cdot,\cdot;\z,\p,\tau) - \mathcal{E}^i_T(\x,t-\tau;\z,\p)\| ^2_{H^{-1}_{s_0}(\R; \bfH_T^{-1/2}(\GammaM))} + \frac{1}{n} \| \mathbf{g}_n(\cdot,\cdot;\z,\p,\tau) \|^2 _{L^2_{s_0}(\R; \bfL_T^{2}(\GammaI))}\, .
  $$
\end{theorem}
\begin{remark}\label{ind_rem}
When we solve the near-field equation we take the polarization vector  $\p$  to be one of the unit vectors 
$\{ \e_1, \e_2, \e_3 \}$. The delay $\tau$ is kept constant for different sampling points $\z$ which in turn are typically on a uniform grid  in an {a priori} chosen test domain.  
Then we use the following
\begin{equation}
\Psi(\z) = 
  \frac{1
  }
  {\sum_{j=1}^3\| \mathbf{g}(\cdot,\cdot;\z,\e_j,{\tau}) \|_{ L^2(\R;\bfL^2_T(\GammaI))}}
  \label{Psidef}
  \end{equation}
as an indicator function for the support of the unknown scatterer. By our theorem we expect $\Psi$ to be small outside the scatterer.  As is usual for the LSM we do not use the norms from the theorem, relying instead on the equivalence of norms in a finite dimensional vector space (arising from the discretization of $\mathbf{g}$). 
\end{remark}

\begin{proof}
We study each case successively.\\
\emph{Case 1.}  Consider a sampling point $\z\in\Omega$. We are trying to solve approximately
$$
(\mathcal{Q} (\mathcal{T}_{\Gamma} ({\Mop}^i_{\chi}\g)))_T = 
\mathcal{E}^i_T(\cdot,\cdot-\tau;\z,\p) \qquad \text{on }\GammaM\times\R \, .
$$
To this end, let us notice that $\mathcal{E}^i(\cdot,\cdot-\tau;\z,\p)$ has no singularity in $\Omega^c$ and is the unique solution of the forward problem with impedance data ${\mathcal{T}}_{\Gamma} (\mathcal{E}^i(\cdot,\cdot-\tau;\z,\p))$, so that $\mathcal{E}^i(\cdot,\cdot-\tau;\z,\p) = \mathcal{Q} ({\mathcal{T}}_{\Gamma} (\mathcal{E}^i(\cdot,\cdot-\tau;\z,\p)))$. In particular $(\mathcal{Q} ({\mathcal{T}}_{\Gamma} (\mathcal{E}^i(\cdot,\cdot-\tau;\z,\p))))_T = \mathcal{E}^i_T(\cdot,\cdot-\tau;\z,\p) $ on $\GammaM\times\R $. Since the range of ${\mathcal{T}}_{\Gamma}{\Mop}_{\chi}: L^{2}_{s_0}(\R;\bfL^2_T(\GammaI))\to H^{p}_{s_0}(\R;\bfH^{1/2}_T(\Gamma))$ is dense, we can approximate ${\mathcal{T}}_{\Gamma} (\mathcal{E}^i(\cdot,\cdot-\tau;\z,\p))$ by a sequence ${\mathcal{T}}_{\Gamma} ({\Mop}^i_{\chi}\g_n )$, where $\g_n \in L^2_{s_0}(\R;\bfL^2_{T}(\GammaI))$ ($n\in\mathbb{N}$). The latter implies that $\mathcal{N}_{\chi}\g_n=(\mathcal{Q} (\mathcal{T}_{\Gamma} ({\Mop}^i_{\chi}\g_n)))_T \to \mathcal{E}^i_T(\cdot,\cdot-\tau;\z,\p)$ in $H^{p}_{s_0}(\R,\bfH^{-1/2}_T(\GammaM))$. Besides, ${\Mop}^i_{\chi}\g_n \in H^p_{s_0}(\R;\bfL^2(\Omega)) $ is divergence free and satisfies $({\Mop}^i_{\chi}\g_n)_{tt} + c_0^2 \curl (\curl {\Mop}^i_{\chi}\g_n) = \bf0$ in $\Omega\times\R$; the latter defines a well-posed interior problem when closed with a tangential condition; hence, from the convergence of  ${\mathcal{T}}_{\Gamma} ({\Mop}^i_{\chi}\g_n )$ towards  ${\mathcal{T}}_{\Gamma} (\mathcal{E}^i(\cdot,\cdot-\tau;\z,\p))$ we conclude that $\Mop^i_{\chi}\g_n$ must remain bounded.
\\
\emph{Case 2.}  Let us next consider a sampling point $\z\in{\OmegaM^{aux}}$ and any sequence $\g_n\equiv \g_{n}(\cdot,\cdot;\z,\p) \in L^2_{s_0}(\R;\bfL^2_{T}(\GammaI))$ such that $\mathcal{N}_{\chi}\g_n\to \mathcal{E}^i_T(\cdot,\cdot-\tau;\z,\p)$ in $H^{p}_{s_0}(\R,\bfH^{-1/2}_T(\GammaM))$  when $n\to\infty $. To derive a contradiction, assume that $\|\g_n\|_{L^2_{s_0}(\R;\bfL^2_{T}(\GammaI))}$ is bounded and take a (weakly) convergent subsequence $\g_n\to\g$ in $L^2_{s_0}(\R;\bfL^2_{T}(\GammaI))$ (here identified with the whole one for simplicity). Then 
$$
\mathcal{N}_{\chi}\g= 
(\mathcal{Q} ({\mathcal{T}}_{\Gamma} ({\Mop}^i_{\chi}\g)) )_T = \mathcal{E}^i_T(\cdot,\cdot-\tau;\z,\p) \qquad \text{on }\GammaM\times\R \, .
$$
In the Fourier-Laplace domain, both $\widehat{{\Mop}}^i_{\chi,s}\hat{\g}_s$ and ${\cal L}[{\cal E}^i(\x,\cdot-\tau;\z,\p)](s) $ are divergence free and solve the equation $\curl\curl\bfu-k_s^2\bfu=\bf0$ in $\OmegaM\setminus \{\z\}$. From this equation and the identity of tangential traces on $\GammaM$, we deduce that they match in ${\OmegaM^{aux}}\setminus\{\z\}$. 
Furthermore, we can use the fact that they are divergence free to rewrite the equation in $\OmegaM\setminus\{\z\}$ as the Helmholtz-like equation $\Delta\bfu +k_s^2\bfu =\bf0$. Then, thanks to the unique continuation principle in $\OmegaM\setminus \{\z\}$, we have that $\widehat{{\Mop}}^i_{\chi,s}\hat{\g}_s={\cal L}[{\cal E}^i(\x,\cdot-\tau;\z,\p)](s)$ in $\OmegaM\setminus \{\z\}$; this leads a contradiction when approaching to $\z$ (recall that $ {\cal L}[{\cal E}^i(\x,\cdot-\tau;\z,\p)](s)$ is singular at $\x=\z$).

\end{proof}


\section{Numerical Development}\label{sec:num}

By choosing appropriate units for time we may assume $c_0=1$ and we do this for the remainder of the paper.

{For each numerical experiment, the Ricker wavelet
\begin{equation}
\label{eq:riku}
\chi(t) = \big(1+2a\, (t-t_0 )^2\big) \, \exp\big(a\, (t-t_0 )^2\big)\, ,
\end{equation}
where $a=-(\pi f_0)^2$, is used as the source modulation function for the incident field (see (\ref{Ei_TD})). Here $f_0$ is the peak frequency of the source, and $t_0 = 1.2/f_0$ is the time delay.}

Figure \ref{fig:signals} shows the Ricker wavelet as a function of time (left panel) and its normalized Fourier spectrum magnitude as a function of wavelength (right panel). The data is shown for frequencies $f_0 = 1$ and $f_0 = 2$ and the propagation media is assumed to be air in which $\epsilon_r = 1$ and $\mu_r = 1$. As described in the appendix, we use a sponge layer (SL) to decrease reflections from the mesh truncation boundary. In the figure, we also visualize the SL thickness values used for both source functions (see the explanation about such sponge layer in the next paragraph and in Appendix \ref{sec:dg}).

\begin{figure}[!ht]
\centering
\resizebox{0.49\textwidth}{!}{\includegraphics{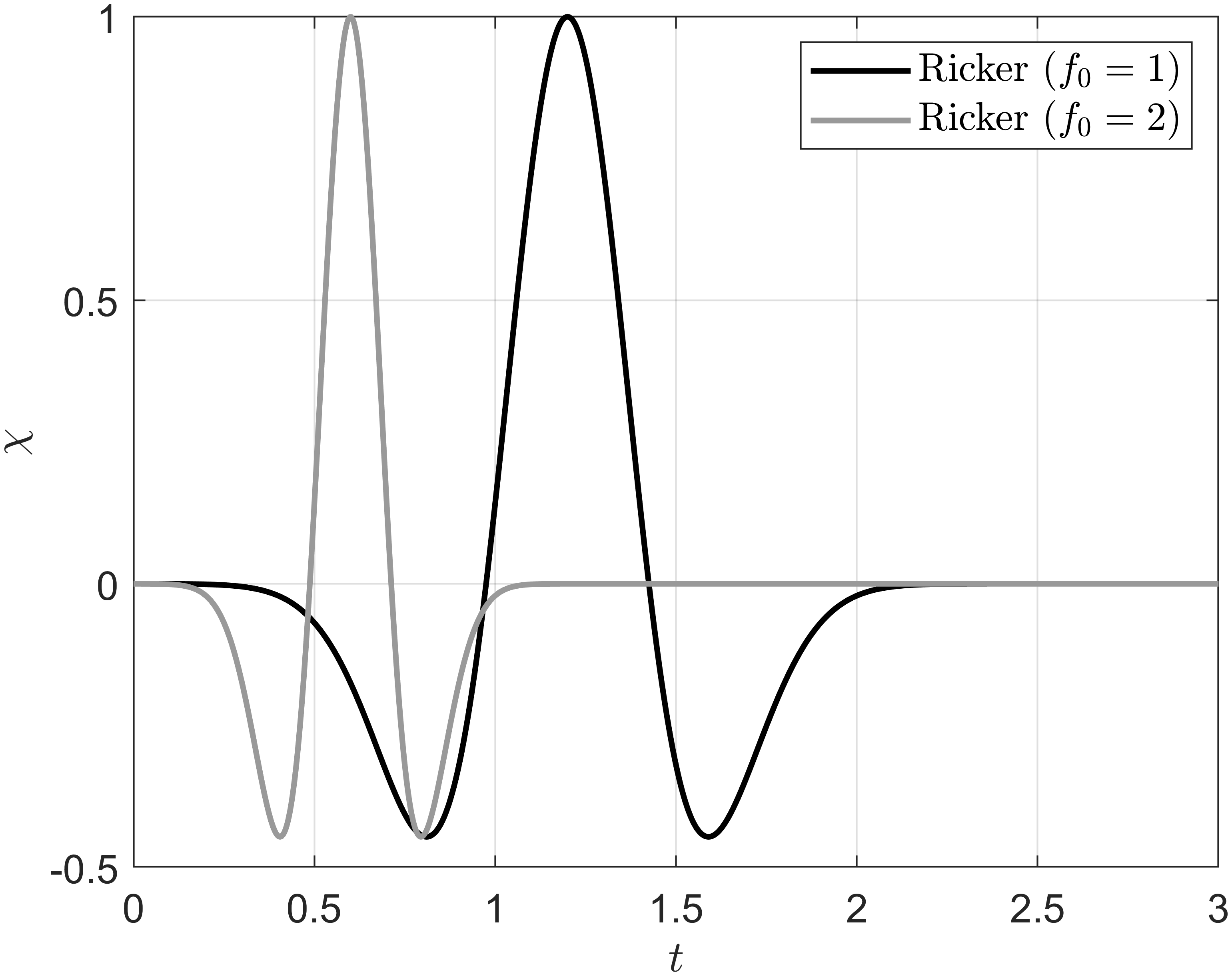}}
\resizebox{0.49\textwidth}{!}{\includegraphics{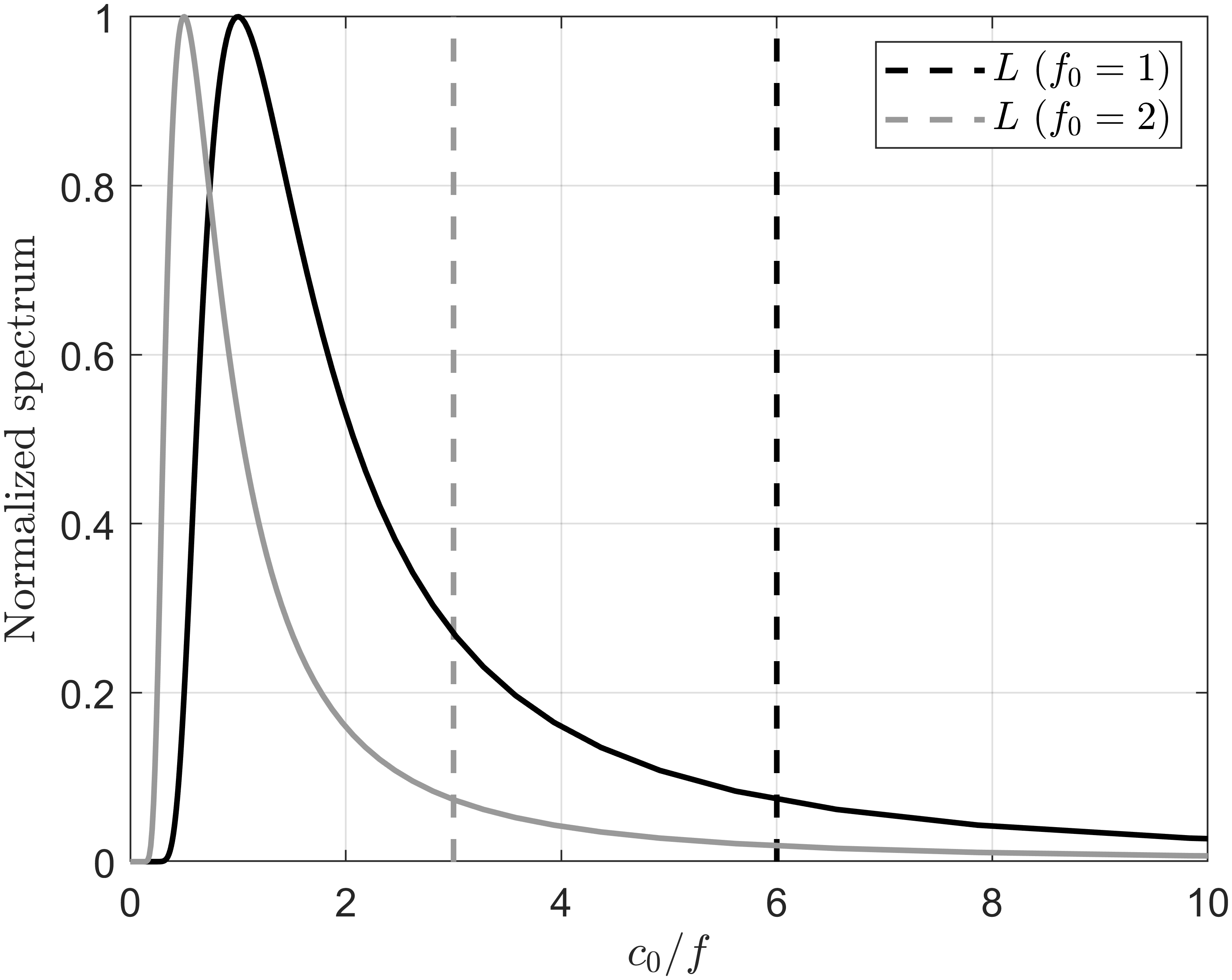}}
\caption{Ricker wavelet for two frequencies in the time domain (left) and the normalized Fourier spectrum magnitude (right). In the graph on the right, also the thickness of the SL is illustrated (dashed vertical lines). }\label{fig:signals}
\end{figure}

\subsection{Algorithmic details for the forward problem}
\label{sec:forward}

To generate synthetic data for testing the TD-LSM we use the nodal Discontinuous Galerkin (DG) method of \cite{hesthaven_warburton_book} for spatial discretization, while the time integration is done by the low-storage explicit Runge-Kutta method \cite{carpenter94}.  More details can be found in the Appendix \ref{sec:dg}.  Here we 
discuss the choice of parameters for the particular tests in this paper.

For the DG simulations the computational domain is divided into tetrahedral elements. Each of the elements contains information on the physical parameters, i.e., the relative permittivity $\epsilon_r$ and relative permeability $\mu_r$ that are assumed to be constant and scalar inside the element but may jump across interfaces between elements. The time step $\Delta t$ for the low-storage Runge-Kutta time stepping method depends on the element size, selected basis order, and the physical parameters, and is computed from
\begin{equation}
\label{eq:CFL}
\Delta t = \min\left\{\frac{h_{\min}^{(\ell)}}{2c^{(\ell)}(N_p+1)^2}\,\, |\,\, \ell=1,\ldots,N_K\right\} ,
\end{equation}
where $c^{(\ell)}=\big(\epsilon_r^{(\ell)}\mu_r^{(\ell)}\big)^{-1/2}$ is the wave speed on the element $\ell$, $N_p$ is the basis order for the DG discretization, $h_{\min}^{(\ell)}$ is the grid size in the element $\ell$, and  $N_K$ is the total number of elements in the grid. In this paper, for the grid size we use smallest distance between two vertices. 

{We choose $\GammaI=\GammaM$ to be the surface of a cube centered at the origin having side length 8. This represents a balance between wanting sources and receivers far from the unknown
scatterers, while also controlling computational cost. We pad this domain by a small layer of width 0.1 to keep sources away from the SL used to terminate the computational domain. Thus, for} each simulation, the main region of interest is an origin-centered cube with a side length of 8.2.
Since the use of a simple Silver-M\"uller absorbing boundary condition on the outer boundary causes unwanted reflections back to the computational domain, we extend the cube size to $(8.2 + 2L)$ {where $L$ is a parameter} and use this extension as the SL that damps the wavefield more efficiently. {The SL is detailed in the Appendix.}  In the current work, the SL is used with the following parameter choices: $L = 6c_0/f_0$, and $\beta_{\max}=10f_0$, where $f_0$ denotes the {peak frequency of the} source {modulation function defined in (\ref{eq:riku}) and $\beta_{\max}$ is the maximum value for the damping coefficient (see Equation (\ref{beta_def}))}. The SL performance is enhanced by coupling it with grid stretching. The stretching coefficient is set to 0.2.

For the computational grid, we use a criterion of 1.5 elements per wavelength (computed using the source peak frequency $f_0$) when constructing the mesh in the main region of interest while in the SL the element size criteria is relaxed to 1.0 element per wavelength. Ninth-order polynomials are applied on each element. The mesh is generated using COMSOL Multiphysics v5.5.

All DG simulation results shown in the following sections are computed using a GPU cluster called Puhti AI, which is part of the CSC's - IT Center for Science supercomputers facilities in Finland. The Puhti AI artificial intelligence partition is equipped with Nvidia Volta V100 graphics cards. Each case is simulated using 20 GPU nodes.

\subsection{Algorithmic details for the inverse problem}\label{sec:IP}

To discretize the near-field equation (\ref{NFE_2}) we discretize the near-field operator $\mathcal{N}_{\chi}$ using quadrature.  To do this we choose a discrete set of source points for the incident fields,   $\{ \x^{\rm{}I}_k |\, k=1,\dots,N_{\rm{}I} \} $ on $\GammaI$, and assume that the scattered field is known (in practice, measured) at measurement points $\{ \x^{\rm{}M}_k |\, k=1,\dots,N_{\rm{}M} \} $ on $\GammaM$.  {In this paper, we usually choose $N_{\rm{}I}<N_{\rm{}M}$ so that, after discretization of the near-field equation described below, we have an overdetermined but ill-conditioned system.} 
 
Integrals over $\GammaI$ are replaced using an associated quadrature formula
\begin{equation}
  \int_{\GammaI} f(\x) \, dA_{\x}
  \approx \sum_{k=1}^{N_{\rm{}I }} \omega_k f(\x^{\rm{}I}_{k}) \, .\label{quadGi}
\end{equation}
As usual the weights are chosen to optimize the degree of precision of the quadrature.

We choose a uniform discretization in time between
$t=0$ and $t=T_{\max}$, yielding
$N_{\rm{}T}+1$ steps $t_j = j \, T_{\mathrm{\max}} / N_{\rm{}T}$
for $j=0,\dots,N_{\rm{}T}$. It is important that $T_{\max}$ is chosen large enough
so that most of the wave energy has passed $\GammaM$
at $t=T_{\max}$. For each source point $\mathbf{x}_k^{\rm{}I}$, we choose two
polarizations $\boldsymbol{p}^{\rm{}I}_{k,\ell}$, $\ell=1,2$, that are mutually orthogonal and orthogonal
to the unit normal at the source point on $\GammaI$ (denoted $\boldsymbol{\nu}(\x^{\rm{}I}_k)$).
Then,  for indices
$(i,j,k) \in \{1,\ldots, N_{\rm{}I}\} \times \{1,\ldots, N_{\rm{}T} \} \times \{1,\ldots, N_{\rm{}M}$\}, we compute the scattering data  $\mathcal{E}_T(\x^{\rm{}M}_i,t_j;\x^{\rm{}I}_k, \boldsymbol{p}^{\rm{}I}_{k,\ell})$
using the GPU accelerated time domain discontinuous Galerkin method described in Section~\ref{sec:forward} and Appendix \ref{sec:dg}. 

For a matrix $\boldsymbol{g} =(g_{k,j})\in\R^{2N_{\rm{}I} \times N_{\rm{}T}}$ of point values of components of $\boldsymbol{g}$ and indices
$(i,j) \in \{1,\ldots, 2N_{\rm{}I}\} \times \{1,\ldots, N_{\rm{}T} \} $, we now approximate the near-field operator by quadrature and collocation as
\[
  ({\cal N}_\chi\boldsymbol{g})(\x^{\rm{}M}_i;t_j)
  \approx   \sum_{l=1}^{N_{\rm{}T}} \sum_{k=1}^{N_{\rm{}I}} \frac{\omega_k}{N_{\rm{}T}}
  \big( g_{k,j} \, \mathcal{E}_T(\x^{\rm{}M}_i,t_j- t_l;\x^{\rm{}I}_k,\boldsymbol{p}^{\rm{}I}_{k,1})  + g_{k+N_{\rm{}I},j} \, \mathcal{E}_T(\x^{\rm{}M}_i,t_j- t_l;\x^{\rm{}I}_k,\boldsymbol{p}^{\rm{}I}_{k,2})\big) \, .
\]
In this sum, the values
$\mathcal{E}_T(\x^{\rm{}M}_i,t_j- t_l;\x^{\rm{}I}_k,\boldsymbol{p}_{k,\ell}^{\rm{}I})= (\boldsymbol{\nu}(\x^{\rm{}M}_i) \times
  \mathcal{E}(\x^{\rm{}M}_i,t_j- t_l;\x^{\rm{}I}_k,\boldsymbol{p}^{\rm{}I}_{k,\ell} ))\times \boldsymbol{\nu}(\x^{\rm{}M}_i)$ are replaced by zero where
$t_j- t_l \leq 0$ or $t_j- t_l > T_{\max}$  (which is consistent with causality and the choice of $T_{\max}$ large enough, respectively).
Here $\boldsymbol{\nu}(\x^{\rm{}M}_k)$ denotes the unit outward normal on $\GammaM$ at $\x^{\rm{}M}_k$. 

Now taking the dot product with two independent tangential vectors $\boldsymbol{p}^{\rm{}M}_{i,\ell}$, $\ell=1,2$, on $\GammaM$ that are orthogonal to $\boldsymbol{\nu}(\x^M_i)$ we obtain the discrete version of the near-field equation
\begin{equation}\label{DNFE}
  \begin{array}{l}  
\displaystyle\sum_{l=1}^{N_{\rm{}T}} \sum_{k=1}^{N_{\rm{}I}} \frac{\omega_k}{N_{\rm{}T}}
  \big( g_{k,j} \, \boldsymbol{p}_{i,q}^{\rm{}M}\cdot\mathcal{E}_T(\x^{\rm{}M}_i,t_j- t_l;\x^{\rm{}I}_k,\boldsymbol{p}^{\rm{}I}_{k,1})  + g_{k+N_{\rm{}I},j} \, \boldsymbol{p}_{i,q}^{\rm{}M}\cdot \mathcal{E}_T(\x^{\rm{}M}_i,t_j- t_l;\x^{\rm{}I}_k,\boldsymbol{p}^{\rm{}I}_{k,2})\big)\\ 
  \qquad \qquad =
  \boldsymbol{p}_{i,q}^{\rm{}M}\cdot \mathcal{E}^i_T(\x^{\rm{}M}_i,t_j;\z,\boldsymbol{p}) \, ,
  \end{array}
\end{equation}
which  is required to hold for $q=1,2$ and $1\leq i\leq N_{\rm{}M}$, $1\leq j\leq N_{\rm{}T}$. Accordingly,  for each sampling point $\z$ and sampling polarization $\boldsymbol{p}$, we need to solve a linear system of $2N_{\rm{}I}N_{\rm{}T}$ unknowns and $2N_{\rm{}M}N_{\rm{}T}$ equations.  
The left-hand side can be represented by a matrix denoted $A$ of dimension $(2N_{\rm{}M} N_{\rm{}T}) \times (2N_{\rm{}I} N_{\rm{}T})$, and this matrix is independent of the sampling point $\z$ or sampling polarization $\boldsymbol{p}$.  {To simulate small measurement errors, a new perturbed matrix is computed via
\[
A^{\rm{}noise}_{\ell,m}=A_{\ell,m}(1+\epsilon_{\rm{}noise} \xi_{\ell,m})
\]
for all $\ell$ and $m$, where $\xi_{\ell,m}$ is a pseudorandom number equidistributed in $[-1,1]$.  We choose $\epsilon_{\rm{}noise}=0.01$ and this results in approximate 0.6\% relative error in the matrix $\ell^2$-norm. The perturbed
system (\ref{DNFE}) using $A^{\rm{}noise}$ in place of $A$} must be solved for each choice of the auxiliary source point $\z$ and auxiliary  polarization $\boldsymbol{p}$. 

In a typical numerical experiment, we use $T_{\max}=20$ and request $N_{\rm{}T}=1250$ time steps from the forward solver (our forward solver uses the time-step in Section~\ref{sec:forward} and this is down-sampled to give the number of points for the inverse solver). For the tests on the next section the measurement grid has $N_{\rm{}M}=96$ and we use $N_{\rm{}I}=54$ and, therefore, the matrix representing the near-field operator is
$ 240,000 \times  135,000  $.  Thus the discrete near-field equation cannot be solved directly by least squares (in addition it is expected to be ill-conditioned since the kernel of the integral operator ${\cal N}_{\chi}$ is smooth). However the action of the matrix can
be computed efficiently using the FFT based method described in \cite{MarmoratEtal}, and so a truncated singular value decomposition
can be computed using the MATLAB function {\tt eigs}. We typically compute 2500 singular vectors and report results for 25, 1000, and 2500 vectors in the truncated SVD expansion.  Then this truncated expansion is used to compute a regularized approximation of $\boldsymbol{g}$ (to handle possible ill-conditioning when large numbers of singular values are used, we also use Tikhonov regularization).  

In all cases the measurement points and source points are located on the surface of the cube $[-4,4]^3$.  This choice represents a balance between the desire to test the inverse solver with  remote measurements and the need for reasonable run times from the forward {solver}.
We consider both measurements and sources located uniformly on this surface as described above.

\begin{figure}[!h]
    \centering
    \resizebox{0.25\textwidth}{!}{\includegraphics{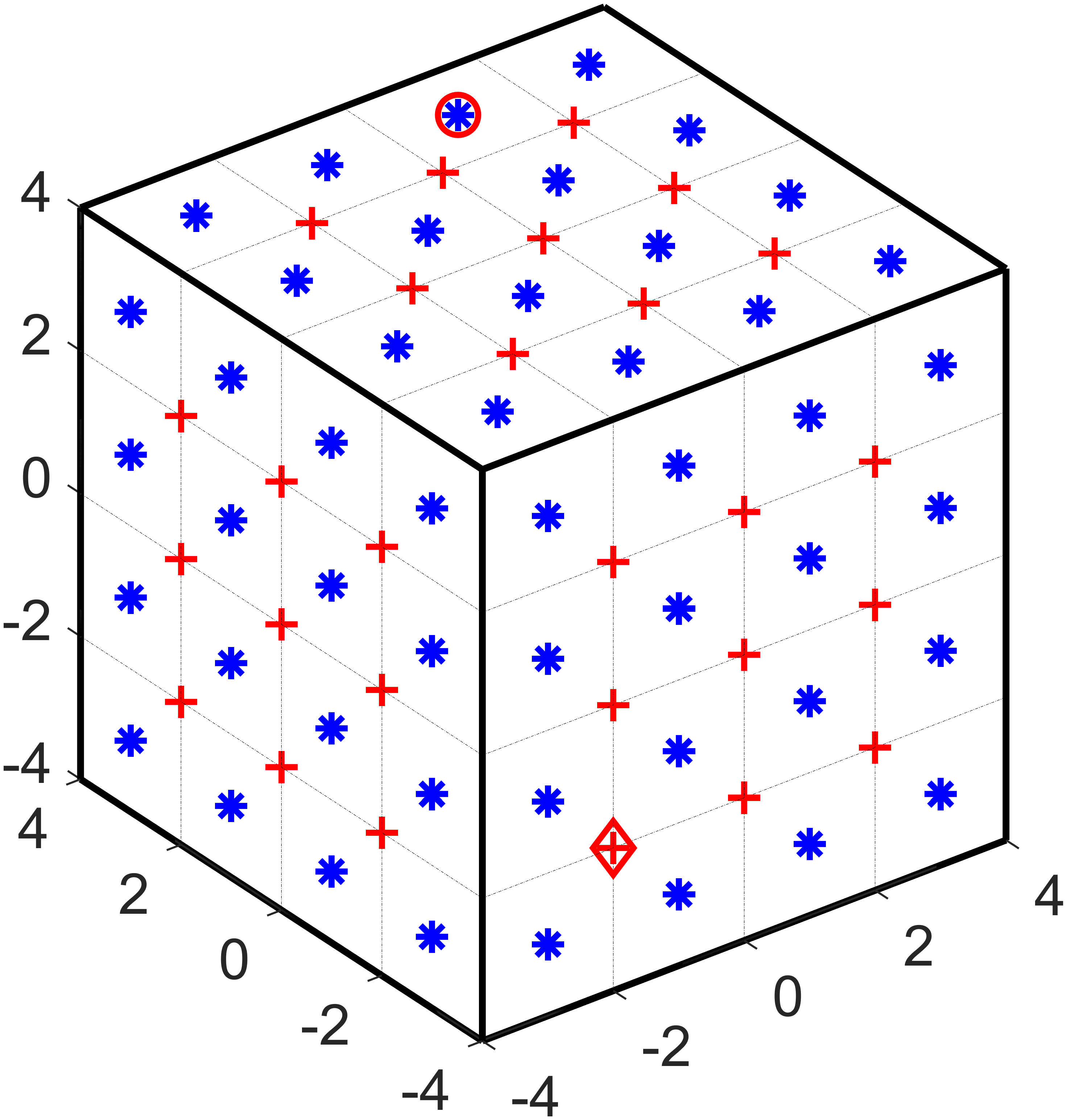}}\hfill
    \resizebox{0.35\textwidth}{!}{\includegraphics{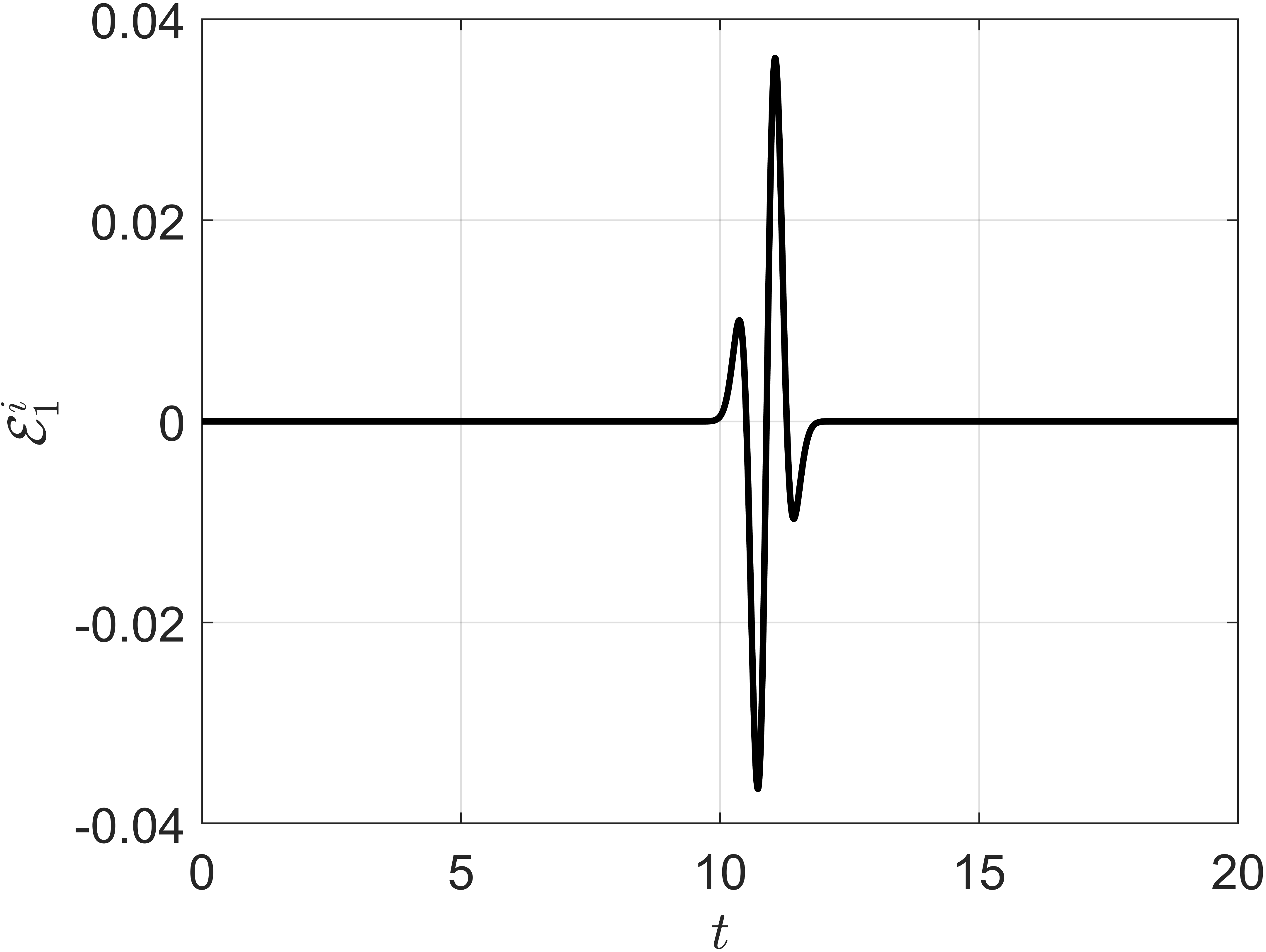}}\hfill
    \resizebox{0.35\textwidth}{!}{\includegraphics{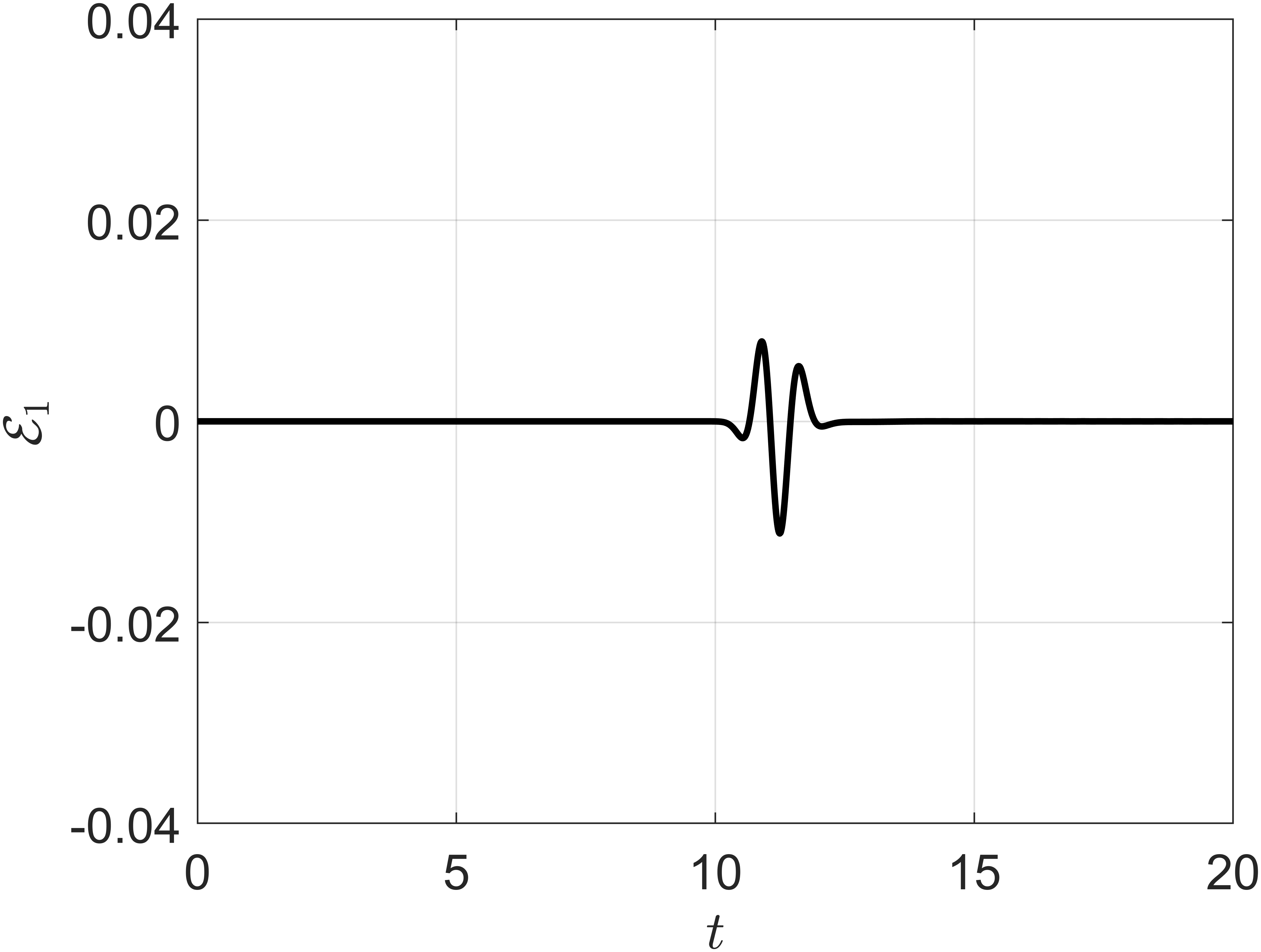}}
    \caption{Source and measurement configuration. Left panel: Source points (marked with red $+$) and measurement points (marked with blue $*$) on the surface $\partial ([-4,4]^3)$. Center panel: The time course of the $\mathcal{E}_1^i$ component of the incident field with polarization $\p=(0,0,1)$ due to the source at the point marked with a red $\diamond$ in the left panel and measured at the measurement point marked with a red $\circ$ in the left panel. Right panel: the scattered field $\mathcal{E}_1$ due to the source used in the left panel, and measured at the indicated measurement point. Here the target is the two cube case with impedance boundary data.}
    \label{fig:meas1}
\end{figure}

We solve the discrete near-field equation for $\z$ on a $41 \times 41\times 41$ uniform sampling grid in the search domain
$[-1.5,1.5]^3$.  Since the search points are on a grid with gridsize $3/41$ we cannot expect to get
 resolution on objects smaller than a few multiplies of this size.  For each source point $\z$, we solve the discrete near-field equation (\ref{DNFE}) for three auxiliary polarizations $\p$ along the axis directions.  The resulting indicator is the reciprocal of the sum of the  $\ell^2$-norm of $\g$ for each polarization (a discrete analogue to $\Psi(\z)$.
 
   In order to present results we draw an isosurface of this function in three dimensional space.  The value of the indicator for this isosurface is taken to be $(\alpha\min_z\Psi(\boldsymbol{z}) + (1-\alpha)\max_z\Psi(\boldsymbol{z}))$ and we choose $\alpha=0.1$
 for most of the numerical experiments except where noted. {The choice of cutoff $\alpha$ works well for the impedance boundary condition, but overestimates the size of a PEC
 object as shown later in Fig.~\ref{fig:2cubes_fullapp_other_cases}.  In practice this parameter
 should have to be ``calibrated'' using computational results for known objects depending on their nature.  This approach is suggested for the related frequency domain LSM for Maxwell's equations in~\cite{ColtonHaddarMonk}.}

We will consider two target geometries:
\begin{description}
    \item[Two cubes:] The target is $\Omega=[-0.75, -0.25]^3\cup
[0.25, 0.75]^3$.  See Fig.~\ref{fig:2cubes_fullapp}, top left panel. 
\item[Four small cubes:] Here the cubes are all translates of the cube $[0,0.2]\times[0,0.2]\times[0,0.2]$.  The cubes are centered at $(0.3,0.2,0.9)$, $(0.3,0.1,0.1)$, $(-0.1,0.2,0)$ and $(0.4,-0.4,-0.5)$. See Fig.~\ref{fig:4cubes_fullapp}, left panel.
    \end{description}
The scattering data is measured on a $4\times 4$ grid on each face of $\GammaM$ and is due to sources on a $3\times 3$ grid on each face of $\GammaI=\GammaM=\partial ([-4,4]^3)$; see Fig.~\ref{fig:meas1}.

\subsection{Numerical results for two cubes}

For the first example, which is analyzed in Theorem~\ref{main_one} in the preceding section, we choose the scatterer to be impenetrable and to have an impedance boundary condition with $\Lambda = \sqrt{2}I_3$ in Equation (\ref{eq:sma_1}) ($\epsilon_r^{{\rm bc}}=2$ and $\mu_r^{{\rm bc}}=1$, see Appendix \ref{sec:dg}). In addition, we set $f_0=1$ for the source wavelet. The measurement setup is shown in Fig.~\ref{fig:meas1} where we also show the incident wave field and the data at a random measurement point located on another face of the measurement surface. 

We solve the inverse problem with the truncated singular value decomposition  using $1000$ singular vectors and with Tikhonov regularization with parameter $10^{-1}$.  In fact, because of the magnitude of the singular values for this problem, the Tikhonov regularization term does not strongly influence the solution.  Surprisingly, the discrete near-field equation is not  ill-conditioned at least for the relatively small number of singular vectors used here.  Note also that the source and measurement setup, having more measurement points than source points, is different to previous work on the TD-LSM and this may improve conditioning. 

\begin{figure}[!h]
    \centering
    \resizebox{0.4\textwidth}{!}{\includegraphics{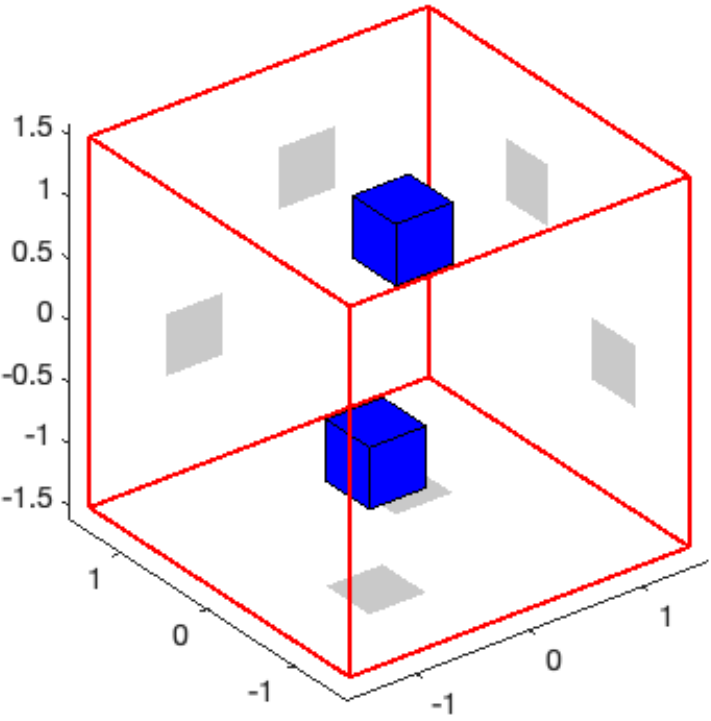}}
    \resizebox{0.41\textwidth}{!}{\includegraphics{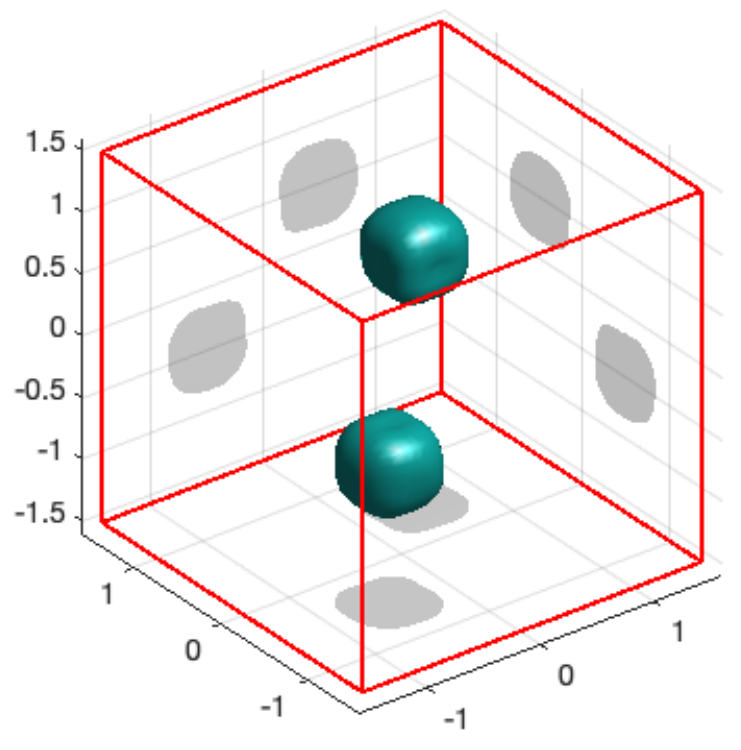}}\\
    \resizebox{0.4\textwidth}{!}{\includegraphics{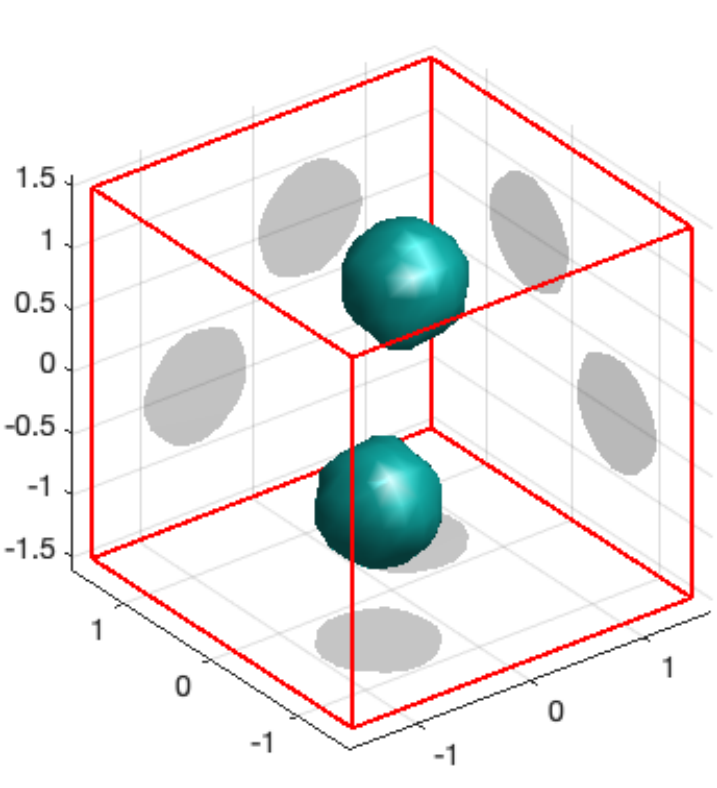}}
    \resizebox{0.41\textwidth}{!}{\includegraphics{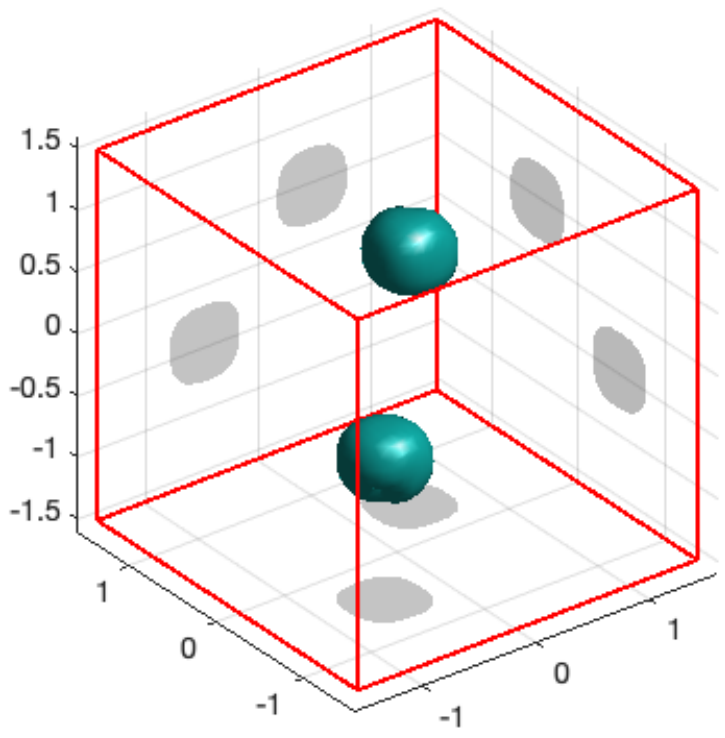}}
    \caption{Top row: left panel shows the exact scatterer, right panel  shows results for impedance boundary conditions. Bottom row: left panel shows results for PEC boundary conditions, and right panel shows results for a 
    penetrable scatterer.  Parameters for the inverse problem are identical for all three problems as discussed in the text. The red 
    cube shows the search region containing the sampling points used to compute $\mathbf{g}$.}
    \label{fig:2cubes_fullapp}
\end{figure}

In Fig.~\ref{fig:2cubes_fullapp} top right panel, we see that the TD-LSM correctly identifies an
approximation to the two cube scatterer.  Although not covered by this paper, a 
similar analysis holds for the TD-LSM applied to a scatterer with a perfect  electric  conductor (PEC) boundary condition.  The corresponding reconstruction is shown in in Fig.~\ref{fig:2cubes_fullapp} bottom left panel. {This can be improved somewhat by adjusting the graphing parameter $\alpha$ but we have kept $\alpha=0.1$ in the figure for consistency.} Finally in Fig.~\ref{fig:2cubes_fullapp} bottom right panel we show the reconstruction when the cubes are penetrable and $\epsilon_r=2I_3$ in each cube. Although the reconstructions are slightly different for each type of scatterer, the 
method gives similar reconstructions using exactly the same inverse {solver} and parameters in all these cases.

Of particular note, the results in Fig.~\ref{fig:2cubes_fullapp} bottom right panel are for a penetrable scatterer.  As we commented in the introduction, the Laplace transform based analysis
used here cannot be applied in this case.  Nevertheless the numerical results indicate that
the TD-LSM can work even for penetrable scatterers. Of course much more testing is needed to make this claim stronger.

The degree of regularization (in our case, {provided both by the Tikhonov regularization and} by the limited  number of singular vectors used in the truncated SVD) effects the quality of the reconstruction.  In Fig.~\ref{fig:2cubes_fullapp_different_nSVD} we repeat the reconstruction of the two cubes using 25 and 2500 singular vectors and otherwise the same parameters as used for Fig.~\ref{fig:2cubes_fullapp}. Both
the highly regularized (25 vectors) and less regularized (2500 vectors with Tikhonov regularization) solution successfully 
locate the scatterers.  The size of the reconstructed images (controlled by the isosurface parameter $\alpha$) are respectively too small and too large.  In the latter case, by changing $\alpha$ we can improve the fit so that the isosurface cutoff needs to be chosen for the particular regularization used.

\begin{figure}[!h]
    \centering
    \resizebox{0.43\textwidth}{!}{\includegraphics{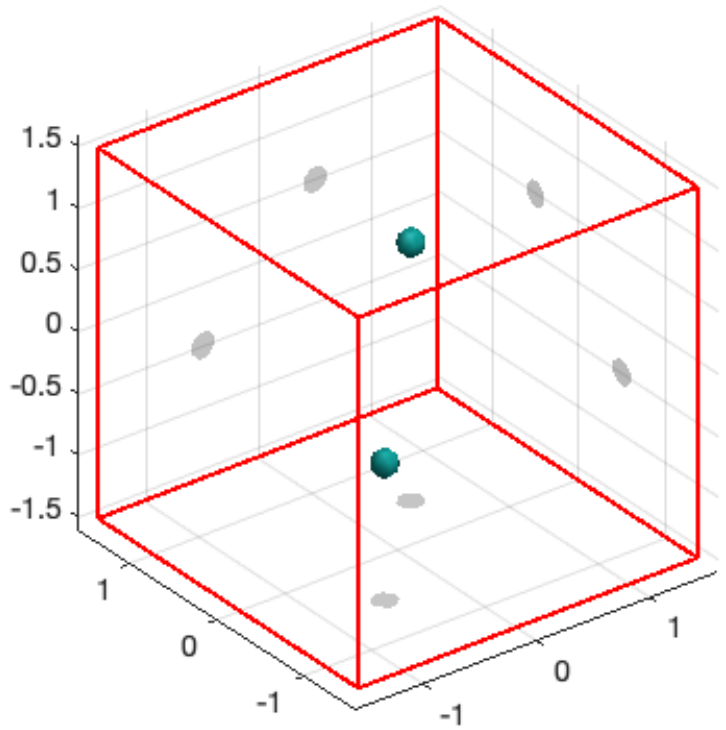}}
    \resizebox{0.43\textwidth}{!}{\includegraphics{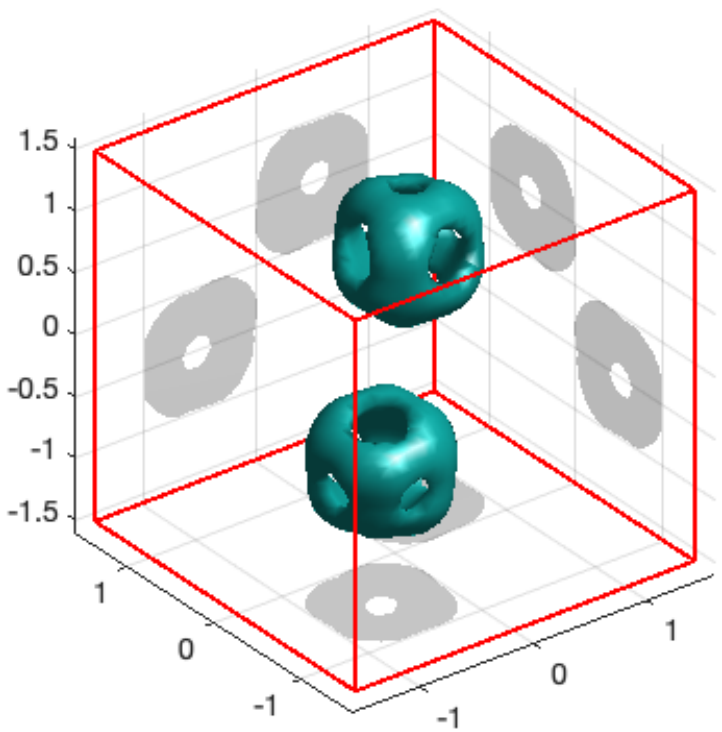}}
\caption{ Left panel shows results for impedance boundary conditions with $25$ singular vectors, and the right-hand side shows the result using 2500 singular vectors. The remaining parameters are identical to those in Fig.~\ref{fig:2cubes_fullapp}.}
    \label{fig:2cubes_fullapp_different_nSVD}
\end{figure}

{For the two cube case the method is not sensitive to the choice of frequencies near $f_0=1$. In the left panel of Fig.~\ref{fig:2cubes_fullapp_other_cases} we show the result of using 
$f_0=2$ and parameter $\alpha=0.25$.  The choice of $\Lambda=\sqrt{2}I_3$ may effect the details of
the reconstruction, but we note that the PEC case can be approximated by choosing $\Lambda$ large
so we do not expect a deterioration of performance for larger $\Lambda$ (see Fig.~\ref{fig:2cubes_fullapp}).  An amusing case is $\Lambda=I_3$ which corresponds to a low order
absorbing boundary condition on $\Gamma$.  Results for this case are shown in Fig.~\ref{fig:2cubes_fullapp_other_cases} right panel. Perhaps surprisingly, the reconstruction compares well to the the reconstruction when $\Lambda=\sqrt{2}I_3$ shown in Fig.~\ref{fig:2cubes_fullapp}.}

\begin{figure}[!h]
    \centering
    \resizebox{0.43\textwidth}{!}{\includegraphics{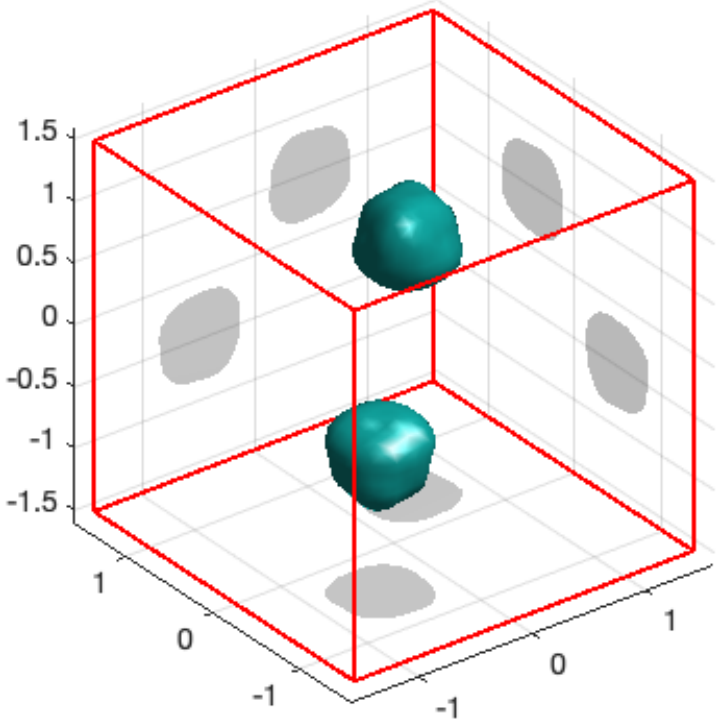}}
    \resizebox{0.43\textwidth}{!}{\includegraphics{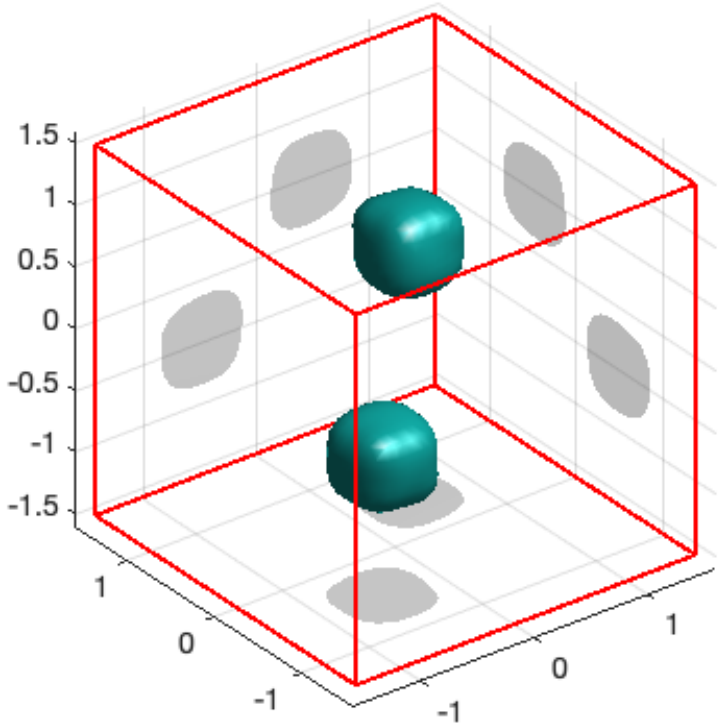}}
\caption{ The left panel shows the 2 cube impedance obstacle with $\Lambda=\sqrt{2}I_3$ with $f_0=2$,  $\alpha=0.25$ and 2,500 singular vectors. The right-hand panel has identical parameters to Fig.~\ref{fig:2cubes_fullapp}
except $\Lambda=I_3$ ($\alpha=0.1$).}
    \label{fig:2cubes_fullapp_other_cases}
\end{figure}

{In a final test we reduce the number of measurement points so that $N_{\rm I}=N_{\rm M}$. The quality of the reconstruction is slightly worse than that in Fig.~\ref{fig:2cubes_fullapp} when $N_{\rm M}>N_{\rm I}$.}
\begin{figure}[!h]
    \centering
\resizebox{0.41\textwidth}{!}{\includegraphics{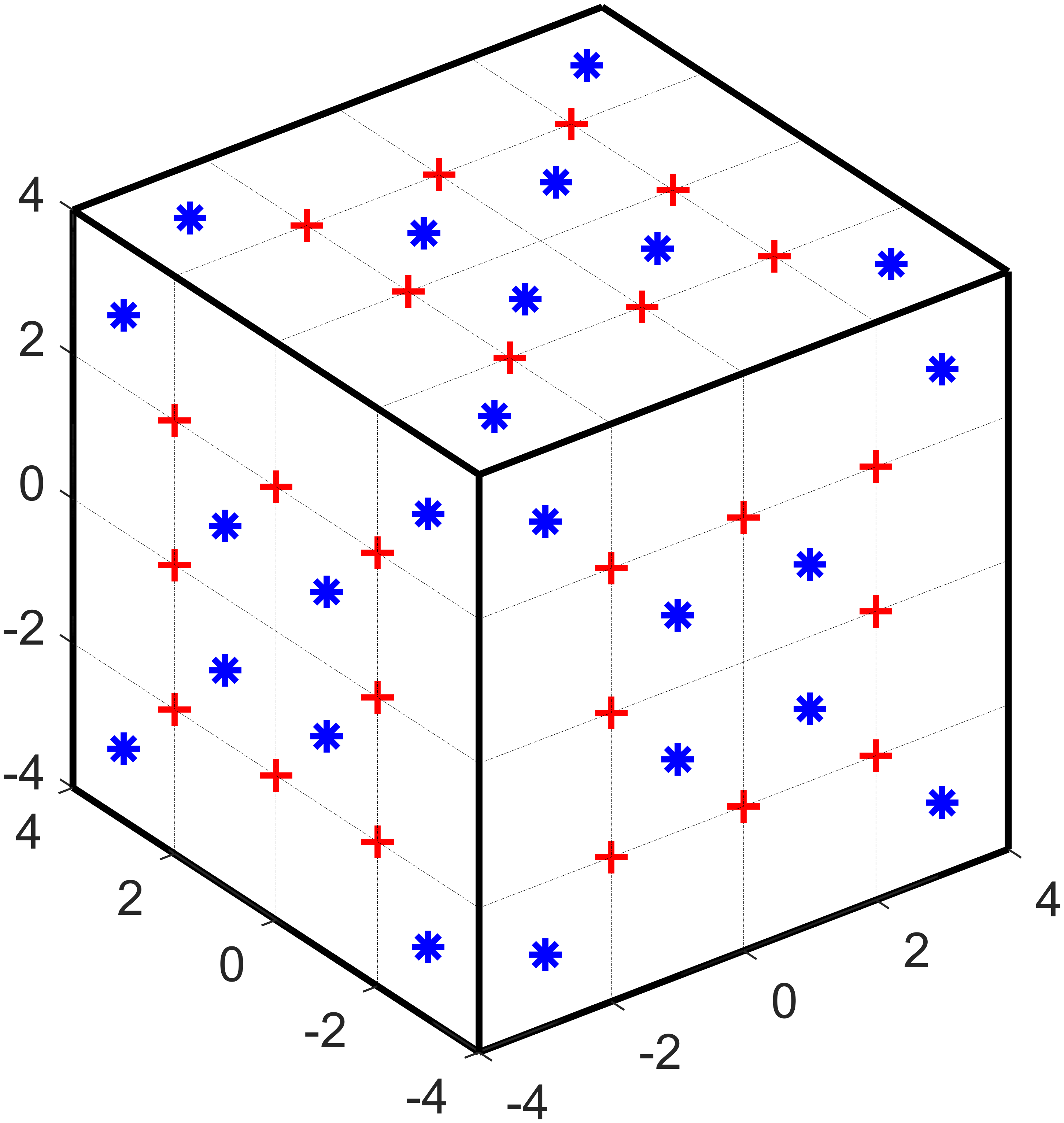}}
    \resizebox{0.43\textwidth}{!}{\includegraphics{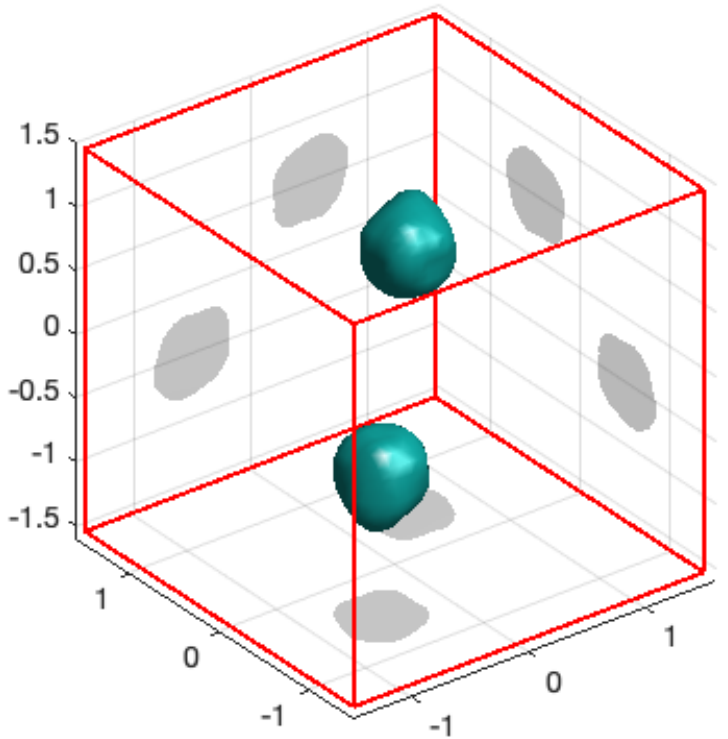}}
\caption{ The left  panel shows the arrangement of
eight transmitters (marked with red $+$) and receivers (marked with blue $\ast$) on each face.  The right  panel shows the 2 cube impedance obstacle using the eight sources and receivers on
each face. }
    \label{fig:2cubes_fullapp_other_cases_a}
\end{figure}

{ All TD-LSM calculations were performed on an Apple Mac Pro having a 2.7 GHz 24-Core Intel Xeon W processor and 384Gb of RAM using Matlab 2021b.  When using 1000 singular vectors for the PEC case with $f_0=2$, the elapsed time for computing the incomplete SVD was 3.8 hours, while the computation of the indicator function depends on the number of auxiliary source points used.  In this paper we use 41 in each direction, for each auxiliary source point computing the indicator function took 0.44 seconds per point or 8.5 hours in total. Similar timings hold for the other cases.}
\subsection{Numerical results for four cubes}
Next we turn to the more difficult reconstruction problem of four small cubes as shown in Fig.~\ref{fig:4cubes_fullapp} left panel.  The scatterers are smaller and two are very close together in comparison to the wavelength of the probing radiation. Nevertheless there is significant scattering: Figure \ref{fig:4cubes_snapshot} shows snapshots of the scattered electric field. The field is shown for time instant, 6.62 (left) and 7.44 (right). The source location is $(-4,\, -2,\, 0)$ and polarization $\bfp =(0,\, 0,\, 1)$ and for the peak frequency we set {$f_0=2$}.

\begin{figure}[!h]
\centering
\resizebox{0.45\textwidth}{!}{\includegraphics{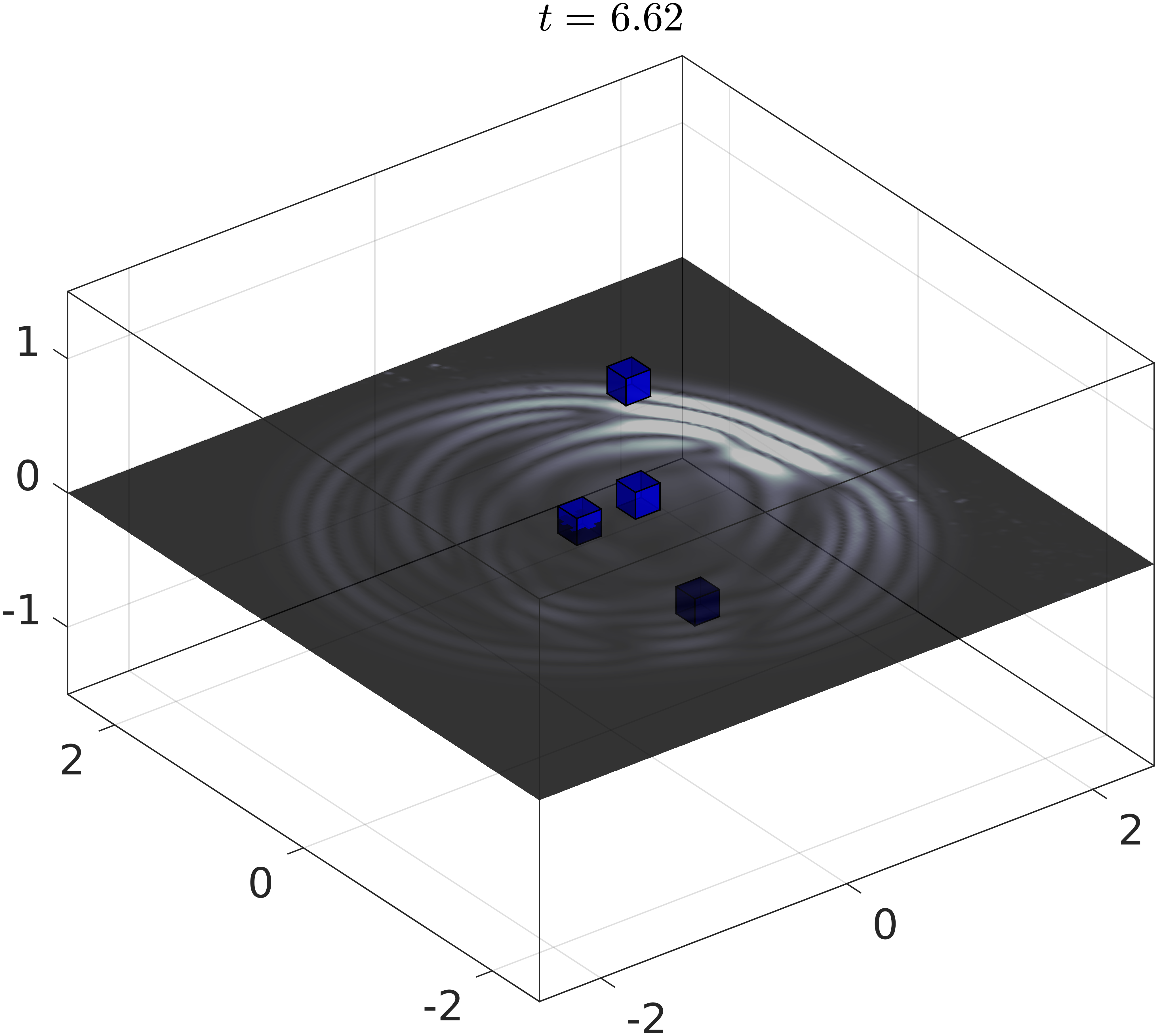}}
\resizebox{0.45\textwidth}{!}{\includegraphics{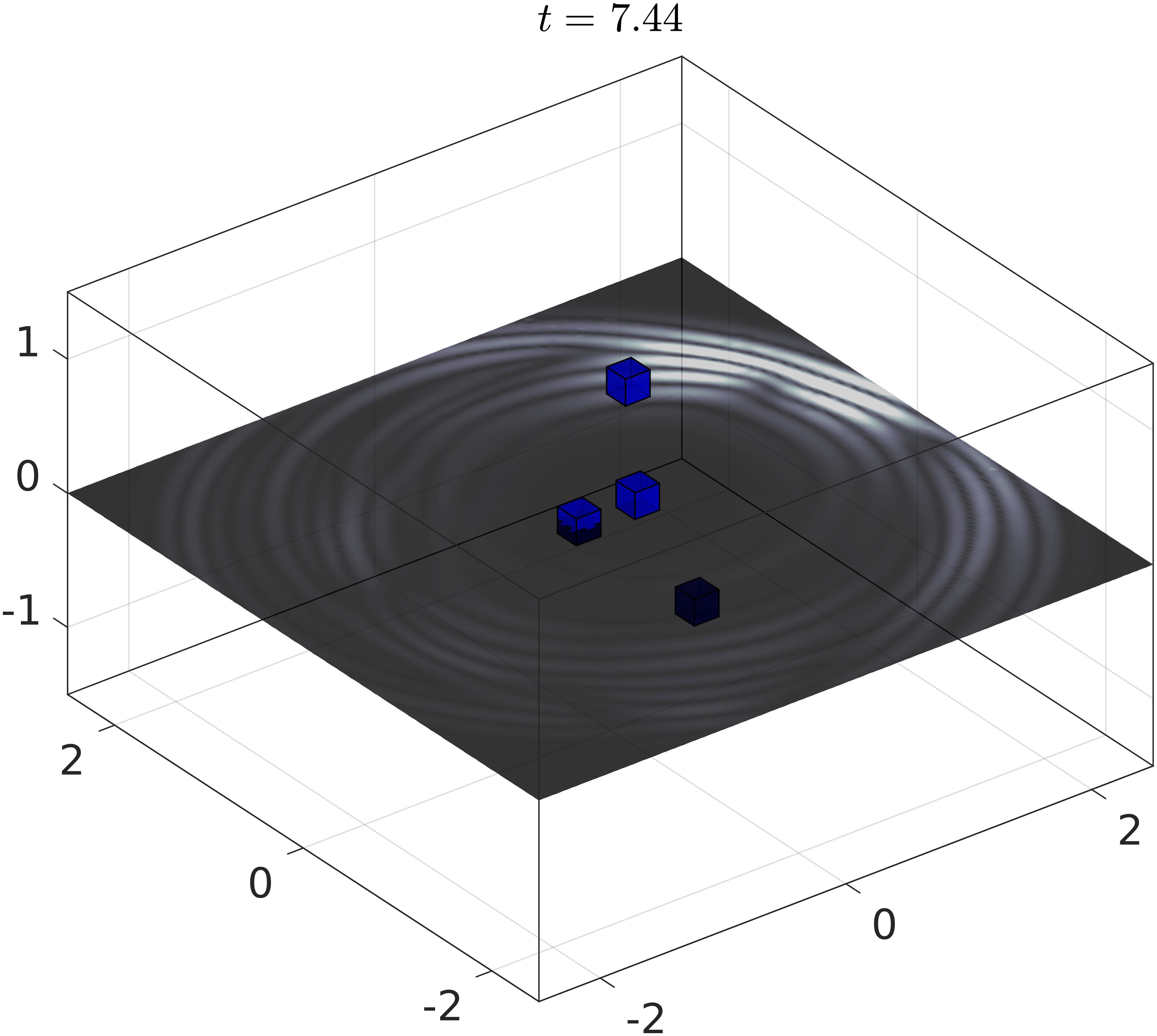}}
\caption{Snapshot of the scattered electric field $\|\calE\|$ for two time instants on a $x_3=0$ plane. The time instant is shown in the figure titles. For $x_1$ and $x_2$, the visualization is limited to $[-2.5,\, 2.5]$ and both graphs show also the geometries and locations of the four scatterers.}
\label{fig:4cubes_snapshot}
\end{figure}

We start by using the same parameters and peak frequency as was used for the previous two cube case in Fig.~\ref{fig:2cubes_fullapp} ($f_0=1$).  The resulting reconstruction is shown in the right panel of Fig.~\ref{fig:4cubes_fullapp}. With the  choice of $\alpha=0.15$ the reconstruction is good, although the choice $\alpha=0.1$ does not reveal the leftmost scatterer.   

\begin{figure}[!h]
    \centering
 \resizebox{0.45\textwidth}{!}{\includegraphics{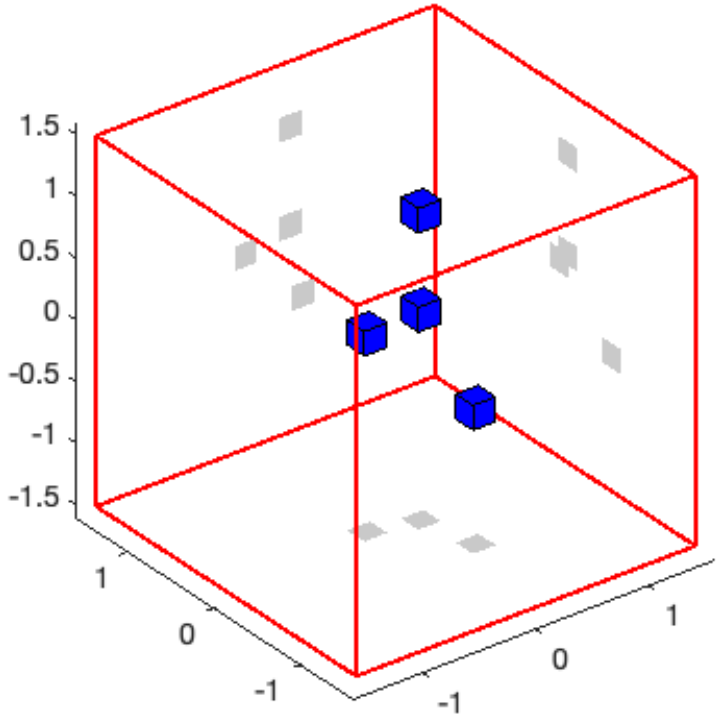}}
 \resizebox{0.45\textwidth}{!}{\includegraphics{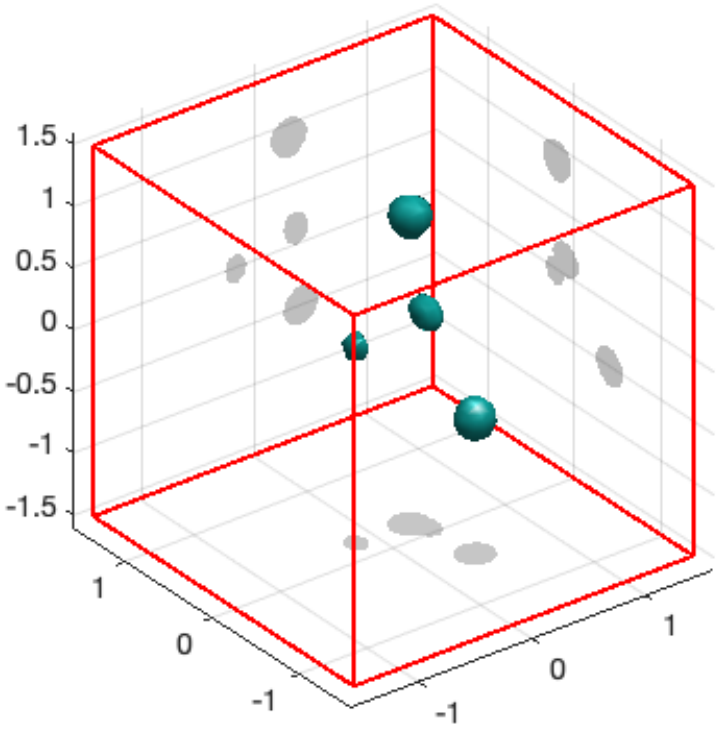}}
    \caption{Left panel shows the exact scatterer.  The right panel shows the reconstruction using the same parameters as for Fig.~\ref{fig:2cubes_fullapp} except $\alpha=0.15$. As before the red box indicates
    the extent of the sampling search region. }
    \label{fig:4cubes_fullapp}
\end{figure}

To test the hypothesis that the peak frequency $f_0$ is too low, we repeat the computation of the
forward data and inverse problem with $f_0=2$.  {Results are shown in Fig.~\ref{fig:4cubes_fullapp_f2}.  In the left panel we use the isosurface parameters\ $\alpha=0.1$.  In the right panel we show the result of {$\alpha=0.27$}. All four scatterers are visible, but now the upper scatterer is more difficult to detect.  In this case the wavelength for the peak frequency is 0.5, and at this frequency the sources are distance 4 wavelengths apart, which is a very coarse array.  Further investigations into the choice of sensor setup and the source frequency are needed.} 

\begin{figure}[!h]
    \centering
 \resizebox{0.45\textwidth}{!}{\includegraphics{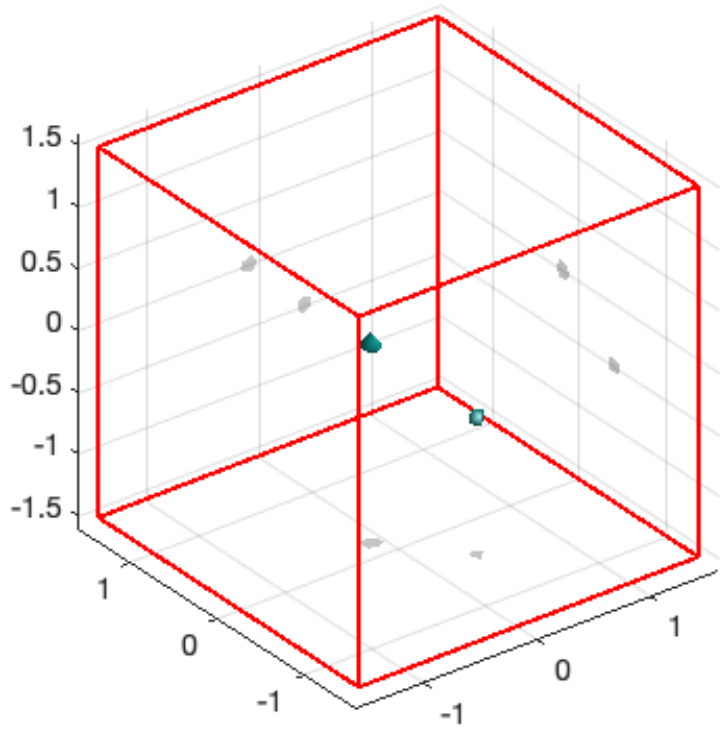}}
\resizebox{0.45\textwidth}{!}{\includegraphics{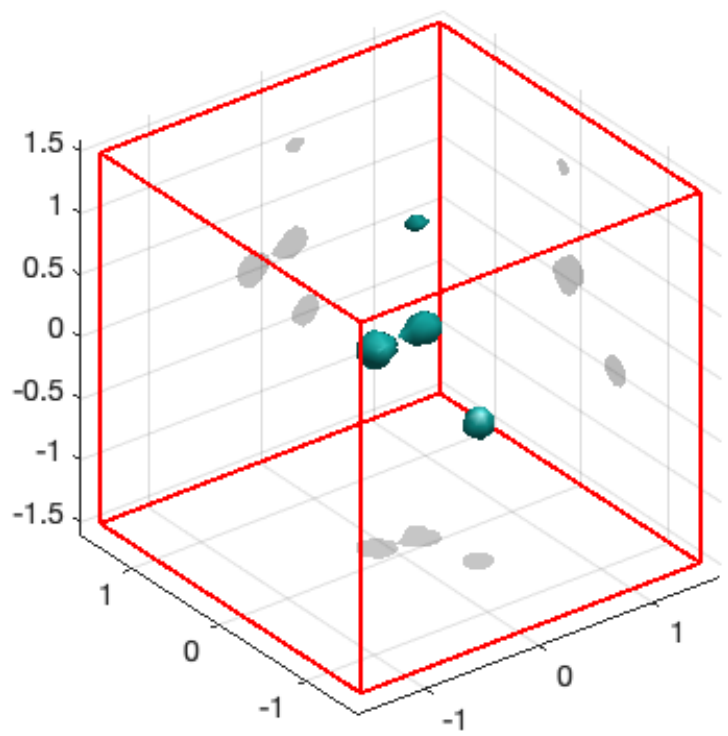}}
    \caption{Left panel: reconstruction of the four cubes with $f_0=2$ and $\alpha=0.1$. Right panel: reconstruction with $\alpha=0.27$.
    The results in the right panel should be compared to the right panel of Fig.~\ref{fig:4cubes_fullapp}.}
    \label{fig:4cubes_fullapp_f2}
\end{figure}

\section{Conclusion}
We have used the TD-LSM for the first time for Maxwell's equations.  By using the Fourier-Laplace transform approach we
have shown that the standard TD-LSM theory can be extended to this case.  We have also provided numerical results that
show the method is successful in reconstructing scatterers.  These numerical results are preliminary, and further work is needed to investigate
limited aperture problems.  

{The TD-LSM is more time consuming than the single frequency LSM  because the matrices involved are much larger so that the time needed for computing the singular vectors
is much longer.  In comparison to multifrequency approaches, the TD-LSM automatically uses data at all frequencies present in the signal, but is again more expensive. For example, in 
\cite{bellis} numerical examples are presented for up to 101 frequencies.  For each frequency, a frequency domain LSM is used to
generate an indicator, and these are combined in a special way.  Since only small frequency domain LSM problems are solved, this approach will still be faster unless a very large number of frequencies are used.  A direct comparison of multi-frequency and time domain LSM methods is a problem for future research}.

{This paper suggests several interesting directions to
improve the practical applicability of the method.  These include the use
of alternative SVD implementations, the use of more general source models, and
an investigation of sensor setup.  This investigation will be facilitated
by tuning the DG forward solver (particularly the SL) to improve the efficiency of generation of forward data.  In addition, we hope to test the inversion scheme on real data.}

An interesting theoretical question is whether the method can be proved for penetrable scatterers, but this
will require a new approach other than using the Fourier-Laplace transform.

\section*{Acknowledgements}

The research of T. L\"ahivaara is supported by the Academy of Finland (the Finnish Centre of Excellence of Inverse Modeling and Imaging) and project 321761. 
The research of P. Monk is partially supported by the US AFOSR under grant number FA9550-20-1-0024. 
The research of V. Selgas is partially supported by the project MTM2017-87162-P of MINECO.
The authors wish also to acknowledge CSC – IT Center for Science, Finland, for computational resources.

\appendix
\section*{Appendix}
\section{The Discontinuous Galerkin forward solver}\label{sec:dg}

Recalling our assumption that $c_0=1$, Maxwell's equations for a regularized {electric} dipole at $\bfx^{\i}$ with polarization $\bfp$ can be written in the conservation form as~\cite[Section 10.5]{hesthaven_warburton_book}
\begin{equation}
\label{eq:cons}
    Q\frac{\partial \q}{\partial t} + \Ddiv \mathcal{F} = -\mathcal{J}.
\end{equation}
Here (\ref{eq:cons}), $Q$ is a block matrix given by
\begin{displaymath}
Q = \left[\begin{array}{cc}
    \epsilon _r \, I_3 & 0  \\
     0 & \mu _r \, I_3
\end{array}\right],\quad 
\q = \left[\begin{array}{c}
     \calE  \\
     \calH 
\end{array}\right],
\end{displaymath}
and
\begin{equation*}
\mathcal{F} = \left[\begin{array}{c}
     \bfFe   \\
     \bfFh  
\end{array}\right] = \left[f_1,\, f_2,\, f_3\right],\quad f_{\ell} = \left[\begin{array}{c}
     -\mathbf{e}_{\ell} \times \calH   \\
     \mathbf{e}_{\ell} \times \calE  
\end{array}\right], 
\end{equation*}
where $\mathbf{e}_{\ell}$ is the ${\ell}$-th Cartesian unit vector. In addition, the right-hand side of (\ref{eq:cons}) is written
\begin{displaymath}
\mathcal{J} = \left[\begin{array}{c}
     \epsilon_r\beta\,\calE + \bfp \,\chi(t)\, \delta_{\bfx^{\i}}  \\
     \mu_r\beta\,\calH  
\end{array}\right].
\end{displaymath}
The coefficient $\beta=\beta(\bfx)$ is used to add additional damping in the zone next to the absorbing condition and $\delta_{\bfx^{\i}}$ is the Dirac delta function at $\bfx^{\i}$. We assume that the materials are isotropic and piecewise homogeneous, so the coefficients corresponding to the relative electric permittivity $\epsilon_r$ and magnetic permeability $\mu_r$ are identified piecewise with real scalars.

The spatial derivatives in (\ref{eq:cons}) are discretized using the nodal discontinuous Galerkin method \cite{hesthaven_warburton_book}, while the time integration is done by the low-storage explicit Runge-Kutta method \cite{carpenter94}. In the discretized version, we assume that the computational domain $\tilde{\Omega}\subset \mathbb{R}^3$ is divided into $N_K$ tetrahedral elements, $\Omega = \bigcup_{k=1}^{N_K} D^k$. The boundary of element $D^k$ is denoted by $\Gamma^k$.  We assume that the elements are aligned with material discontinuities. Furthermore, for any element $D^k$ the superscript `$-$' refers to interior information while `$+$' refers to exterior information.

We multiply (\ref{eq:cons}) by a local test function $\phi^k$ and integrate by parts twice to obtain an elementwise variational formulation
\begin{equation}
\label{strong_form}
  \int_{D^k}  \left(Q\frac{\partial \q^k}{\partial t} 
    + \Ddiv\mathcal{F}+\mathcal{J}\right) \phi^k\, d\x
  =  \int_{\Gamma^k} 
  \bfnu \cdot\left(\mathcal{F}^- - \mathcal{F}^*\right) \phi^k\, dA\, ,
\end{equation}
$\q^k$ is the restriction of $\q$ to the element $D^k$ and $\mathcal{F}^*$ is the numerical flux across neighbouring element interfaces.  For the numerical flux $\mathcal{F}^*$ along the normal $\bfnu$, we use the upwind \cite{HESTHAVEN2002186}
\begin{eqnarray}
\label{eq:fe}
 \bfnu \cdot \left(\bfFe^- - \bfFe^* \right) &=&  \frac{1}{Z^++Z^-}\bfnu \times\left(Z^+\ddcalH - \bfnu\times\ddcalE\right),\\
\label{eq:fh}
\bfnu \cdot \left(\bfFh^- - \bfFh^* \right) &=&  -\frac{1}{Y^++Y^-}\bfnu \times\left(Y^+\ddcalE + \bfnu\times\ddcalH\right),
\end{eqnarray}
where $Z^{\pm}=\frac{1}{Y^{\pm}}=\sqrt{\frac{\mu_r^{\pm}}{\epsilon_r^{\pm}}}$,  $\ddcalH = \dcalH$, and $\ddcalE = \dcalE$.

In this work, we apply impedance and perfect electric conductor (PEC) boundary conditions. On the exterior boundary, the PEC condition is recovered from (\ref{eq:fe}) and (\ref{eq:fh}) by setting $\epsilon_r^+ = \epsilon_r^-$, $\mu_r^+ = \mu_r^-$,
\begin{equation}
  \calE^+ =  -\calE^-,\quad \text{and}\quad \calH^+ = \calH^-.
\end{equation}
The impedance boundary condition is obtained from (\ref{eq:fe}) and (\ref{eq:fh}) by setting $\epsilon_r^+ = \epsilon_r^- = \epsilon_r^{{\rm bc}}$, $\mu_r^+ = \mu_r^- = \mu_r^{{\rm bc}}$, and
\begin{equation}
\label{eq:sma}
  \calE^+ = \calH^+ = \boldsymbol{0} \, .
\end{equation}
The impedance boundary condition reduces to the Silver-M\"uller absorbing (SMA) by setting parameters $\epsilon_r^{{\rm bc}} = \epsilon_r^-$ and $\mu_r^{{\rm bc}} = \mu_r^-$ (i.e. the physical values of the interior element). Unfortunately, the SMA condition is not perfect and some unwanted reflections will happen at the outflow boundaries if the incoming wave is not parallel with the boundary. In this paper, we couple the absorbing boundary condition with a sponge layer (SL) that damps the wave. To do so, the variable $\beta$ introduced in (\ref{eq:cons}) is non-zero for the regions next to the SMA condition. Moreover, the SL is coupled with the grid stretching.

To illustrate the functioning of the SL, let us consider an example in which the layer is applied on one coordinate axis only. Now, for a node coordinate $x_1^{(\ell)}\in[x_1^{(0)},\ x_1^{(0)}+L]$, where $x_1^{(0)}$ denotes a starting location of the SL and $L$ its thickness, the grid stretched coordinate $\hat{x}_1^{(\ell)}$ is defined as
\begin{eqnarray}
    \hat{x}_1^{(\ell)} = x_1^{(0)} + \big(x_1^{(\ell)} - x_1^{(0)} \big) \left(1+g_{\max}\Big(\frac{x_1^{(\ell)} - x_1^{(0)}}{L}\Big)^3\right),
\end{eqnarray}
where $g_{\max}$  denotes the maximum value given for the grid stretching. Similarly the $\beta$ value at $\hat{x}_1^{(\ell)}$ in the SL is
\begin{eqnarray}\label{beta_def}
  \beta(\hat{x}_1^{(\ell)}) = \beta_{\max}\, \Big(\frac{\hat{x}_1^{(\ell)} - x_1^{(0)}}{L(1+g_{\max})}\Big)^3\, ,
\end{eqnarray}
where $\beta_{\max}$ is the maximum value given for the damping coefficient.

The current version of the wave solver is written in the C/C++ programming language and is integrated with the Open Concurrent Compute Abstraction (OCCA) \cite{medina2014occa} library and message passing interface to enable parallel computations both on CPU and GPU clusters. Currently, the solver uses only constant order basis functions and the computational load between different elements is balanced by the parMetis software \cite{Karypis}.

{Due to the assumption of using a magnetic dipole as a source, and the fact that the current DG-based wave solver assumes an electric dipole, we use the magnetic field ${\cal H}$ as data for the inverse solver.  Because of the constant coefficients in Maxwell's equations, this corresponds to the electric field due to a magnetic dipole.}


\begin{thebibliography}{37}
\providecommand{\natexlab}[1]{#1}
\providecommand{\url}[1]{\texttt{#1}}
\expandafter\ifx\csname urlstyle\endcsname\relax
  \providecommand{\doi}[1]{doi: #1}\else
  \providecommand{\doi}{doi: \begingroup \urlstyle{rm}\Url}\fi

\bibitem[Bamberger and Duong(1986)]{BHa86}
A.~Bamberger and T.~Ha Duong.
\newblock Formulation variationnelle espace-temps pour le calcul par potentiel
  retarde de la diffraction d'une onde acoustique ({I}).
\newblock \emph{Math. Meth. Appl. Sci.}, 8:\penalty0 405--435, 1986.

\bibitem[Belkebir and Saillard(2001)]{belk01}
K.~Belkebir and M.~Saillard.
\newblock Special section: Testing inversion algorithms against experimental
  data.
\newblock \emph{Inv. Prob.}, 17:\penalty0 1565-1571, 2001.

\bibitem[Belkebir and Saillard(2004)]{belk04}
K.~Belkebir and M.~Saillard.
\newblock Testing inversion algorithms against experimental data: Inhomogeneous
  targets.
\newblock \emph{Inv. Prob.}, 21:\penalty0 S1--S3, 2004.

\bibitem[Cakoni et~al.(2011)Cakoni, Colton, and Monk]{ccm-book}
F.~Cakoni, D.~Colton, and P.~Monk.
\newblock \emph{The Linear Sampling Method in Inverse Electromagnetic
  Scattering}, volume~80 of \emph{CBMS}.
\newblock SIAM, Philadelphia, 2011.

\bibitem[Cakoni et~al.(2021{\natexlab{a}})Cakoni, Meng, and Xiao]{cak21}
F.~Cakoni, S.~Meng, and J.~Xiao.
\newblock A note on transmission eigenvalues in electromagnetic scattering
  theory.
\newblock \emph{Inverse Problems and Imaging}, 15:\penalty0 999--1014,
  2021{\natexlab{a}}.

\bibitem[Cakoni et~al.(2021{\natexlab{b}})Cakoni, Monk, and Selgas]{CMS-21}
F.~Cakoni, P.~Monk, and V.~Selgas.
\newblock Analysis of the linear sampling method for imaging penetrable
  obstacles in the time domain.
\newblock \emph{Anal. PDE}, 14:\penalty0 667--688, 2021{\natexlab{b}}.
\newblock DOI: 10.2140/apde.2021.14.667.

\bibitem[Carpenter and Kennedy(1994)]{carpenter94}
M.H. Carpenter and C.A. Kennedy.
\newblock \emph{Fourth-order {2N}-storage {R}unge-{K}utta schemes}.
\newblock Technical report, NASA-TM-109112, 1994.

\bibitem[Chen et~al.(2010)Chen, Haddar, Lechtleiter, and Monk]{CHLM10}
Q.~Chen, H.~Haddar, A.~Lechtleiter, and P.~Monk.
\newblock A sampling method for inverse scattering in the time domain.
\newblock \emph{Inv. Prob.}, 26, 2010.
\newblock 085001 (17pp).

\bibitem[Chen(2018)]{XChen18}
X.~Chen.
\newblock \emph{Computational Methods for Electromagnetic Inverse Scattering}.
\newblock Wiley-IEEE Press, 2018.

\bibitem[Colton and Kirsch(1996)]{col96}
D.~Colton and A.~Kirsch.
\newblock A simple method for solving inverse scattering problems in the
  resonance region.
\newblock \emph{Inv. Prob.}, 12:\penalty0 383--93, 1996.

\bibitem[Colton and Kress(2019)]{ColtonKress}
D.~Colton and R.~Kress.
\newblock \emph{Inverse Acoustic and Electromagnetic Scattering Theory}.
\newblock Springer--Verlag, New York, 4th edition, 2019.

\bibitem[Colton et~al.(2002)Colton, Haddar, and Monk]{ColtonHaddarMonk}
D.L. Colton, H.~Haddar, and P.~Monk.
\newblock The linear sampling method for solving the inverse electromagnetic
  scattering problem.
\newblock \emph{SIAM J. Sci. Comput.}, 24:\penalty0 719--731, 2002.

\bibitem[Donato and Morabito(2020)]{dido19}
L.~Di Donato and A.~F. Morabito.
\newblock \emph{Special issue ``Microwave imaging and electromagnetic inverse
  scattering problems''}, volume~5.
\newblock Mdpi AG, 2020.

\bibitem[Dorn and Lesselier(2010)]{dorn10}
O.~Dorn and D.~Lesselier.
\newblock Introduction to the special issue on electromagnetic inverse
  problems: Emerging methods and novel applications.
\newblock \emph{Inv. Prob.}, 26, 2010.
\newblock Art. No. 070201.

\bibitem[G\"odel et~al.(2010)G\"odel, Nunn, Warburton, and Clemens]{godel10}
N.~G\"odel, N.~Nunn, T.~Warburton, and M.~Clemens.
\newblock Scalability of higher-order discontinuous {G}alerkin fem computations
  for solving electromagnetic wave propagation problems on {GPU} clusters.
\newblock \emph{IEEE Transactions on Magnetics}, 46\penalty0 (8):\penalty0
  3469--3472, 2010.

\bibitem[Guo et~al.(2013)Guo, Monk, and Colton]{GuoEtal}
Y.~Guo, P.~Monk, and D.~Colton.
\newblock Toward a time domain approach to the linear sampling method.
\newblock \emph{Inv. Prob.}, 29\penalty0 (095016), 2013.

\bibitem[Guzina et~al.(2010)Guzina, Cakoni, and Bellis]{bellis}
B.~Guzina, F.~Cakoni, and C.~Bellis.
\newblock On the multi-frequency obstacle reconstruction via the linear
  sampling method.
\newblock \emph{Inverse Problems}, 26\penalty0 (125005), 2010.

\bibitem[Haddar et~al.(2014)Haddar, Lechleiter, and Marmorat]{MarmoratEtal}
H.~Haddar, A.~Lechleiter, and S.~Marmorat.
\newblock An improved time domain linear sampling method for {R}obin and
  {N}eumann obstacles.
\newblock \emph{Applicable Analysis}, 93:\penalty0 369--390, 2014.

\bibitem[Haddar et~al.(2016)Haddar, Hiptmair, Monk, and Rodriguez]{CIME}
H.~Haddar, R.~Hiptmair, P.~Monk, and R.~Rodriguez.
\newblock \emph{Computational electromagnetism}.
\newblock Lecture Notes in Mathematics, 2148, Fondazione CIME/CIME Foundation
  Subseries. Springer, 2016.
\newblock Notes from the CIME School held in Cetraro, June 9-14, 2014. Edited
  by A. Berm\'udez de Castro and A. Valli.

\bibitem[Hesthaven and Warburton(2002)]{HESTHAVEN2002186}
J.~S. Hesthaven and T.~Warburton.
\newblock Nodal high-order methods on unstructured grids: {I}. time-domain
  solution of {M}axwell's equations.
\newblock \emph{Journal of Computational Physics}, 181\penalty0 (1):\penalty0
  186--221, 2002.

\bibitem[Hesthaven and Warburton(2007)]{hesthaven_warburton_book}
J.~S. Hesthaven and T.~Warburton.
\newblock \emph{Nodal Discontinuous {G}alerkin Methods: Algorithms, Analysis,
  and Applications}.
\newblock Springer, 2007.

\bibitem[Karypis et~al.(1997)Karypis, Schloegel, and Kumar]{Karypis}
G.~Karypis, K.~Schloegel, and V.~Kumar.
\newblock Parmetis: {P}arallel graph partitioning and sparse matrix ordering
  library.
\newblock 1997.

\bibitem[Kl\"{o}ckner et~al.(2009)Kl\"{o}ckner, Warburton, Bridge, and
  Hesthaven]{Klockner09}
A.~Kl\"{o}ckner, T.~Warburton, J.~Bridge, and J.~S. Hesthaven.
\newblock Nodal discontinuous {G}alerkin methods on graphics processors.
\newblock \emph{J. Comput. Phys.}, 228\penalty0 (21):\penalty0 786-882,
  2009.

\bibitem[L\"ahivaara et~al.(2019)L\"ahivaara, Malehmir, Pasanen,
  K\"arkk\"ainen, Huttunen, and Hesthaven]{tl19}
T.~L\"ahivaara, A.~Malehmir, A.~Pasanen, L.~K\"arkk\"ainen, J.~M.~J. Huttunen,
  and J.~S. Hesthaven.
\newblock Estimation of groundwater storage from seismic data using deep
  learning.
\newblock \emph{Geophysical Prospecting}, 67\penalty0 (8):\penalty0 2115--2126,
  2019.

\bibitem[Li and Hesthaven(2014)]{LI2014915}
Jichun Li and J.~S. Hesthaven.
\newblock Analysis and application of the nodal discontinuous {G}alerkin method
  for wave propagation in metamaterials.
\newblock \emph{Journal of Computational Physics}, 258:\penalty0 915--930,
  2014.

\bibitem[Litman and Crocco(2009)]{lit09}
A.~Litman and L.~Crocco.
\newblock Testing inversion algorithms against experimental data: 3d targets.
\newblock \emph{Inv. Prob.}, 25, 2009.
\newblock Art. No. 020201.

\bibitem[Lubich(1994)]{Lubich}
Ch. Lubich.
\newblock {On the multistep time discretization of linear initial-boundary
  value problems and their boundary integral equations}.
\newblock \emph{Numer. Math.}, 67:\penalty0 365--389, 1994.

\bibitem[Medina et~al.(2014)Medina, St-Cyr, and Warburton]{medina2014occa}
D.S. Medina, A.~St-Cyr, and T.~Warburton.
\newblock {OCCA}: {A} unified approach to multi-threading languages.
\newblock \emph{arXiv preprint arXiv:1403.0968}, 2014.

\bibitem[Melander et~al.(2020)Melander, Str{\o{}}m, Pind, Engsig-Karup, Jeong,
  Warburton, Chalmers, and Hesthaven]{Melander}
A.~Melander, E.~Str{\o{}}m, F.~Pind, A.~Engsig-Karup, C.-H. Jeong, T.~Warburton,
  N.~Chalmers, and J.~S. Hesthaven.
\newblock Massive parallel nodal discontinuous {G}alerkin finite element method
  simulator for room acoustics.
\newblock \emph{International Journal of High Performance Computing
  Applications}, 2020.
\newblock URL \url{http://infoscience.epfl.ch/record/279868}.

\bibitem[Monk(2003)]{MonkBook}
P.~Monk.
\newblock \emph{Finite Element Methods for {M}axwell's Equations}.
\newblock Oxford University Press, Oxford, 2003.

\bibitem[Monk and Selgas(2016)]{monksel16}
P.~Monk and V.~Selgas.
\newblock An inverse acoustic waveguide problem in the time domain.
\newblock \emph{Inverse Problems}, 32:\penalty0 055001, 2016.

\bibitem[Prunty and Snieder(2019)]{prunty19}
A.~Prunty and R.~Snieder.
\newblock Theory of the linear sampling method for time-dependent fields.
\newblock \emph{Inv. Prob.}, 35, 2019.
\newblock 055003, DOI: 10.1088/1361-6420/ab0ccd.

\bibitem[Rudin(1973)]{Rudin1973}
W.~Rudin.
\newblock \emph{Functional Analysis}.
\newblock McGraw-Hill, 1973.

\bibitem[Sayas(2016)]{Sayas2016}
F.J. Sayas.
\newblock \emph{Retarded Potentials and Time Domain Boundary Integral
  Equations: A Road Map}.
\newblock Springer-Verlag, 2016.

\bibitem[Vodev(2018)]{Vodev-18}
G.~Vodev.
\newblock High-frequency approximation of the interior {D}irichlet-to-{N}eumann
  map and applications to the transmission eigenvalues.
\newblock \emph{Anal. PDE}, 11:\penalty0 213-236, 2018.

\bibitem[Vodev(2021)]{Vodev-21}
G.~Vodev.
\newblock Semiclassical parametrix for the {M}axwell equation and applications
  to the electromagnetic transmission eigenvalues.
\newblock arXiv: https://arxiv.org/abs/2102.08662, 2021.

\bibitem[Wilcox et~al.(2010)Wilcox, Stadler, Burstedde, and
  Ghattas]{WILCOX20109373}
L.~C. Wilcox, G.~Stadler, C.~Burstedde, and O.~Ghattas.
\newblock A high-order discontinuous {G}alerkin method for wave propagation
  through coupled elastic-acoustic media.
\newblock \emph{Journal of Computational Physics}, 229\penalty0 (24):\penalty0
  9373--9396, 2010.

\end{thebibliography}
\end{document}